\documentclass[a4paper,12pt]{amsart}

\usepackage{amsmath, amsthm, amssymb, amsfonts}
\usepackage[normalem]{ulem}
\usepackage{hyperref}

\usepackage{xypic, mathtools}
\usepackage{verbatim} 
\usepackage{longtable} 

\theoremstyle{plain}
\newtheorem{thm}{Theorem}

\newtheorem{prop}[thm]{Proposition}

\theoremstyle{definition}

\theoremstyle{remark}

\newcommand{\QMod}{\mathrm{QMod}}
\newcommand{\QQ}{\mathbb{Q}}

\newcommand{\e}{\mathsf{1}}

\newcommand{\BC}{{\mathbb{C}}}

\newcommand{\BH}{{\mathbb{H}}}

\newcommand{\BP}{{\mathbb{P}}}
\newcommand{\BQ}{{\mathbb{Q}}}

\newcommand{\BZ}{{\mathbb{Z}}}

\newcommand{\pt}{{\mathsf{p}}}

\newcommand{\CA}{{\mathcal A}}

\newcommand{\CE}{{\mathcal E}}
\newcommand{\CF}{{\mathcal F}}

\newcommand{\Fp}{{\mathfrak{p}}}

\newcommand{\Fz}{{\mathfrak{z}}}

\newcommand{\ra}{{\ \longrightarrow\ }}

\newcommand{\blangle}{\big\langle}
\newcommand{\brangle}{\big\rangle}

\DeclareMathOperator{\Tr}{Tr}

\DeclareMathOperator{\Hilb}{Hilb}
\DeclareMathOperator{\Aut}{Aut}

\DeclareMathOperator{\id}{id}

\newcommand{\Mbar}{{\overline M}}
\newcommand{\M}{{\overline M}}

\newcommand{\ev}{\mathop{\rm ev}\nolimits}
\newcommand{\p}{\mathbb{P}}

\usepackage{tikz}

\begin{document}

\baselineskip=16pt
\parskip=5pt

\title[Curve counting on $K3\times E$]
{Curve counting on $K3 \times E\, $, the Igusa cusp form $\chi_{10}\, $,  
and descendent integration}
\author {G. Oberdieck}
\address {ETH Z\"urich, Department of Mathematics}
\email {georgo@math.ethz.ch}
\author{R. Pandharipande}
\address{ETH Z\"{u}rich, Department of Mathematics}
\email {rahul@math.ethz.ch}
\date{June 2015}
\maketitle

\begin{abstract}
Let $S$ be a nonsingular projective $K3$ surface.
Motivated by the study of the Gromov-Witten theory of the
Hilbert scheme of points of $S$,
we conjecture a formula  
for the Gromov-Witten theory (in all curve classes) of
the Calabi-Yau 3-fold $S\times E$ where $E$ is an elliptic curve.
In the primitive case, our conjecture is expressed in terms of
the Igusa cusp form $\chi_{10}$ and matches a prediction 
via
heterotic duality by Katz, Klemm, and Vafa.
In imprimitive cases, our conjecture suggests a
new structure for the complete theory of descendent integration
for $K3$ surfaces. 
Via the Gromov-Witten/Pairs correspondence, a
conjecture for the reduced stable pairs theory of 
$S\times E$ is also presented. Speculations about
the motivic stable pairs theory of $S\times E$ are made.

The reduced Gromov-Witten theory of the 
Hilbert scheme of points of $S$ is much richer than $S\times E$.
The 2-point function of $\text{Hilb}^d(S)$ determines a
matrix with trace equal to the partition function of $S\times E$.
A conjectural form for the full matrix is given.

\end {abstract}

\setcounter{tocdepth}{1} 
\tableofcontents

\setcounter{section}{-1}

\section{Introduction}
Let $S$ be a nonsingular projective $K3$ surface, and let
$E$ be a nonsingular elliptic curve.
The 3-fold
$$X=S \times E$$
has trivial canonical bundle, and hence is Calabi-Yau.
Let 
$$\pi_1: X \rightarrow S\ , \ \ \ \pi_2:X\rightarrow E $$
denote the projections on the respective factors.
Let 
$$\iota_{S,e}: S \rightarrow X, \ \ \ \iota_{E,s}:E \rightarrow X$$
be inclusions of the fibers of $\pi_2$ and $\pi_1$
over points $e\in E$ and $s\in S$ respectively. We will often
drop the subscripts $e$ and $s$.

Let $\beta\in \text{Pic}(S) \subset H_2(S,\mathbb{Z})$ be a class
which is positive (with respect to any ample polarization), 
and let $d\geq 0$ be an integer.
The pair $(\beta,d)$
determines a class in
$H_2(X,\mathbb{Z})$ by
$$(\beta,d)= \iota_{S*}(\beta)+ \iota_{E*}(d[E])\ .$$
The moduli space
of stable maps 
$\overline{M}_{g}\big(X, (\beta,d)\big)$  
from connected genus $g$ curves to $X$ representing
the class $(\beta,d)$
is of virtual dimension 0.
However, because $S$ is holomorphic symplectic, the
virtual class vanishes{\footnote{See \cite{gwnl} for discussion 
of the virtual class for stable maps to $K3$ surfaces.}},
$$\Big[\overline{M}_{g}\big(X, (\beta,d)\big)\Big]^{vir}=0\ .$$  

The Gromov-Witten theory of $X$ is only interesting after
{\em reduction}. The reduced class
$\big[\overline{M}_{g}\big(X, (\beta,d)\big)\big]^{red}$
is of dimension 1. The elliptic curve $E$ acts on
$\overline{M}_{g}\big(X,(\beta,d)\big)$ with orbits of
dimension 1. The moduli space $\overline{M}_{g}\big(X, (\beta,d)\big)$
may be viewed virtually as a finite union of $E$-orbits.
The basic enumerative question here is {to count the
number of these $E$-orbits}.

The counting of the $E$-orbits is defined mathematically by the
following construction.
 Let $\beta^\vee\in H^2(S,\mathbb{Q})$
be {\em any} class satisfying
\begin{equation}\label{fwwf}
\langle \beta, \beta^\vee \rangle =1
\end{equation}
with respect to the intersection pairing on $S$.
For $g\geq0$, we define
\begin{equation}
\label{fccf}
\mathsf{N}^X_{g,\beta,d} = 
\int_{[\overline{M}_{g,1}(X,(\beta,d))]^{red}} \text{ev}_1^*\Big(\pi_1^*(\beta^\vee)
\cup \pi_2^*([0])\Big)\ ,
\end{equation}
where $0\in E$ is the zero of the group law.
The invariant $\mathsf{N}^X_{g,\beta,d}$ {\em is} the
virtual count of $E$-orbits discussed above.
Because of orbifold issues and the possible non-integrality of $\beta^\vee$,
$$\mathsf{N}^X_{g,\beta,d}\in \mathbb{Q}\ .$$

We will prove $\mathsf{N}_{g,\beta,d}^X$
does not depend upon the choice of $\beta^\vee$ satisfying \eqref{fwwf}.
The count $\mathsf{N}_{g,\beta,d}^X$ is invariant
under deformations of $S$ for which $\beta$ remains algebraic.
By standard arguments \cite{gwnl,CTC}, $\mathsf{N}_{g,\beta,d}^X$
then depends upon $S$ and $\beta$ {\em only} via
the norm square $$2h-2=\langle \beta, \beta \rangle$$ and the 
divisibility $m(\beta)$.
The count $\mathsf{N}_{g,\beta,d}^X$ is independent
of the complex structure of $E$. 
The notation
\begin{equation}\label{vvttt}
\mathsf{N}_{g,m,h,d}^X= \mathsf{N}_{g,\beta,d}^X
\end{equation}
will be used.

We conjecture here four basic properties of the 
reduced Gromov-Witten counts
$\mathsf{N}_{g,\beta,d}^X$:
\begin{enumerate}
\item[(i)] a closed formula for their generating series
in term of the Igusa cusp form $\chi_{10}$ in case $\beta$ is primitive,
\item[(ii)] a reduction rule expressing the invariants for imprimitive
$\beta$ in terms of the primitive cases determined by (i),
\item[(iii)] a Gromov-Witten/Pairs correspondence governing
              the reduced stable pairs invariants of $X$,
\item[(iv)] a precise formula relating $\mathsf{N}_{g,\beta,d}^X$
to the reduced genus 0 Gromov-Witten invariants of
the Hilbert scheme  $\text{Hilb}^d(S)$ of $d$ points of
the $K3$ surface $S$ in case $\beta$ is primitive. 
\end{enumerate}

In the $d=0$ case, the counts $\mathsf{N}^X_{g,\beta,0}$ specialize to
the basic integrals
\begin{equation}\label{kkv}
\mathsf{N}^X_{g,\beta,0} = 
\int_{[\overline{M}_{g}(S,\beta)]^{red}} (-1)^g \lambda_g
\end{equation}
of the reduced Gromov-Witten theory of $K3$ surfaces.{\footnote{Here,
$\lambda_g$ is the top Chern class of the Hodge bundle.}}
The integrals \eqref{kkv} are governed by the
Katz-Klemm-Vafa conjecture proven in \cite{PT2}.
Formula (i) specializes to the
Jacobi form of the KKV formula. Formulas for
BPS counts of $S\times E$ associated to primitive 
curve classes $\beta\in H_2(S,\mathbb{Z})$ are predicted
in \cite[Section 6.2]{KKV} via heterotic duality.
 After suitable interpretation of the Gromov-Witten
theory \eqref{fccf}, our formulas (i) match those of \cite{KKV}.

The imprimitive structure (ii) is new and takes a surprisingly
different form from the standard 3-fold Gromov-Witten multiple cover
theory.
In fact, conjecture (ii) suggests a
new structure for the complete theory of descendent integration
for $K3$ surfaces:
\begin{enumerate} 
\item[(v)] We conjecture  a reduction rule expressing the 
descendent integrals
$$\left\langle \prod_{i=1}^n \tau_{\alpha_i}(\gamma_i) \right\rangle_{g,\beta}^S=
\int_{[\overline{M}_{g,n}(S,\beta)]^{red}} \prod_{i=1}^n \psi_i^{\alpha_i}
\cup \text{ev}_i^\ast(\gamma_i)\ , \ \ \ \gamma_i\in H^*(S,\mathbb{Q})$$
 for imprimitive
$\beta$ in terms of the primitive cases.
\end{enumerate}
By \cite[Theorem 4]{MPT}, the descendent integrals in the
primitive cases are known to be coefficients of quasi-modular forms.

The GW/P correspondence (iii) for $X$ is straightforward
to conjecture. Because reduced
theories are considered, the correspondence here is not directly
a special case of the standard GW/P correspondence for
3-folds \cite{MNOP1,PT1}.

In \cite{BP,OPH,OPLC}, a triangle of parallel equivalences
relating the Gromov-Witten and Donaldson-Thomas theory
of $\mathbb{C}^2 \times \mathbb{P}^1$ to the quantum cohomology
of $\text{Hilb}^d(\mathbb{C}^2)$ was established.
Equivalences relating the counting theories of $S\times \mathbb{P}^1$
and the quantum cohomology of $\text{Hilb}^d(S)$ were expected.
However,
the conjectured formula (iv) 
relating $\mathsf{N}_{g,\beta,d}^X$
to the reduced genus 0 Gromov-Witten invariants of
$\text{Hilb}^d(S)$ is subtle:
an interesting correction term appears.

The 2-point function in the reduced genus 0 Gromov-Witten theory
of $\text{Hilb}^d(S)$ studied in \cite{GO} underlies (iv) and
motivates the entire paper. 
An interesting speculation
which emerges concerns the 3-fold geometry
\begin{equation}\label{lwwl}
S \times \mathbb{P}^1 \, / \,  \{S_0 \cup S_\infty\}
\end{equation}
relative to the $K3$-fibers over $0,\infty \in \mathbb{P}^1$:
\begin{enumerate} 
\item[(vi)] For primitive $\beta\in \text{Pic}(S)$, we conjecture a form 
for the matrix of relative invariants of the geometry \eqref{lwwl}.
\end{enumerate}
The reduced Gromov-Witten invariants of $S\times E$ arise
as the trace of the matrix (vi).

The precise statements of the above conjectures are given in Sections 4 and 5.
Conjecture A of Section 4.1 covers both (i) and (iv). Conjectures B, C, and D of Sections 4.2-4.4 
correspond to (ii), (v), and (iii) respectively. Conjectures E, F, and G of Section 5 address 
(vi) via the reduced Gromov-Witten theory of $\text{Hilb}^d(S)$.
Conjectures E and F were first proposed in \cite{GO} in a different but equivalent form.
Conjecture G is a direct Hilbert scheme / stable pairs correspondence (again
with a correction term).

We conclude the paper with speculations about the motivic stable
pairs invariants of $S\times E$. The theory should
simultaneously refine the Igusa cusp form $\chi_{10}$ and
generalize the formula of \cite{KKP}.

\vspace{10pt}
\noindent{\em Acknowledgements}.
We thank J. Bryan, S. Katz, A. Klemm,  D. Maulik, A. Pixton, R. Thomas,
and B. Szendroi for many conversations over the years
about the Gromov-Witten theory of $K3$ surfaces. We thank M. Raum for discussions about $\chi_{10}$.
The paper was partially written while both authors were  attending
the summer school
{\em Modern trends in Gromov-Witten theory} at the Leibniz Universit\"at
Hannover 
organized by O. Dumitrescu and N. Pagani in September 2014.

G.O. was supported by the grant SNF-200021-143274.
R.P. was partially supported by 
grants SNF-200021-143274 and ERC-2012-AdG-320368-MCSK.


\section{Rubber geometry}
\subsection{Definition}
Let $R$ be the 1-dimensional rubber target
obtained from the relative geometry
$$\mathbb{P}^1\, /\, \{0,\infty\}$$
after quotienting by the scaling action.
Let $Y$ be the straight rubber over
the $K3$ surface $S$,
$$ Y = S \times {R}\ .$$
The moduli space of stable maps to rubber,
$$\overline{M}^\bullet_{g}\big(Y,(\beta,d)\big)_{\nu,\nu^\vee}^{\sim}\, ,$$
has reduced virtual dimension 0. Here:
\begin{enumerate}
\item[(i)] the superscript $\bullet$
indicates the domain curve may be disconnected
(but no connected components are mapped to points),
\item[(ii)] $\beta\in \text{Pic}(S)$ and $d\geq 0$ is the
degree over $R$,
\item[(iii)] the relative conditions over $0$ and $\infty$
of the rubber are specified by partitions of $d$ weighted
by the cohomology of $S$,
\item[(iv)] $\nu$ and $\nu^\vee$ are
dual cohomology weighted partitions.{\footnote{Let $\{\gamma_i\}$
be a basis of $H^*(S,\mathbb{Z})$, let and $\{\gamma_i^\vee\}$ be the
dual basis. If 
$\nu= \{ (\nu_j,\gamma_{i_j}) \}$, then $\nu^\vee =
\{ (\nu_j, \gamma_{i_j}^\vee) \}$.}}
\end{enumerate} 
We define
\begin{equation}\label{pssp}
\widetilde{\mathsf{N}}_{g,\beta,d}^{Y\bullet}(\nu,\nu^\vee) =
\int_{[\overline{M}^\bullet_{g}(Y,(\beta,d))_{\nu,\nu^\vee}]^{red}} 1\ .
\end{equation}
The definition of $\widetilde{\mathsf{N}}_{g,\beta,d}^{Y\bullet}(\nu,\nu^\vee)$
requires no insertion as in \eqref{fccf}.

\subsection{Disconnected invariants of $S\times E$}
In order to relate the integrals \eqref{fccf} and \eqref{pssp},
we must consider the disconnected Gromov-Witten theory
of $$X=S\times E\ .$$ Let
$\overline{M}_{g,1}^\bullet\big(X, (\beta,d)\big)$  
be the moduli space of stable maps from
from possibily disconnected genus $g$ curves 
to $X$ (with 
no connected components mapped to points)
representing
the class $(\beta,d)$. 
After reduction, the moduli space is of dimension 2.
For $\beta^\vee\in H_2(S,\mathbb{Z})$ satisfing \eqref{fwwf},
we define
\begin{equation}
\label{fxxf}
\mathsf{N}^{X\bullet}_{g,\beta,d} = 
\int_{[\overline{M}^\bullet_{g,1}(X,(\beta,d))]^{red}} 
\text{ev}_1^*\Big(\pi_1^*(\beta^\vee)
\cup \pi_2^*([0])\Big)\ ,
\end{equation}
where $0\in E$ is the zero of the group law as before.

Because of the holomorphic symplectic form of $S$,
the stable maps with two connected components
mapping nontrivially to $S$ contribute 0 to \eqref{fxxf}.
Hence, the only nontrivial contibutions to \eqref{fxxf}
come from stable maps with a single marked  connected component
mapping nontrivially to $S$ and possibly other connected components
contracted over $S$.
By standard vanishing considerations,
all connected components contracted over $S$ must be
of genus 1. After evaluating the contracted contributions, we 
obtain the following relation:

\begin{prop} For all $g\geq 0$ and $\beta\in \text{\em{Pic}}(S)$, the disconnected and
connected counts for $X$ satisfy
$$\sum_{d\geq 0} \mathsf{N}^{X\bullet}_{g,\beta,d}\, \tilde{q}^{d}
=  \frac{\sum_{d\geq 0} \mathsf{N}^{X}_{g,\beta,d}\, \tilde{q}^{d}}
{\prod_{n>0} (1-\tilde{q}^n)^{24}}\ .$$
\end{prop}

\subsection{Relating $X$ and $Y$}
Consider the degeneration of $E$ to a nodal rational curve $C$.
The degeneration, 
$$ X= S \times E \ \ \leadsto \ \ S \times C\ , $$
leads to a formula for $\mathsf{N}^{X\bullet}_{g,\beta,d}$
in terms of the relative geometry
$$ S \times \mathbb{P}^1 \, / \, \{S_0 \cup S_\infty\}\ .$$
Then, using standard rigidification of the rubber and the
divisor axiom, we obtain the relation:
\begin{equation}
\label{kssk}
 \mathsf{N}^{X\bullet}_{g,\beta,d} =
\Big[\sum_{\nu \in \mathcal{P}(d)} \Fz(\nu) 
\, \widetilde{\mathsf{N}}^{Y\bullet}_{\beta,d}
(\nu,\nu^\vee)\,  u^{2 \ell(\nu)}\Big]_{u^{2g-2}}\ .
\end{equation}
Here, $\mathcal{P}(d)$ is the set of all cohomology 
weighted partitions of $d$ with respect to a
fixed basis $\{\gamma_i\}$ of $H^*(S,\mathbb{Z})$.
The rubber series on the right side of \eqref{kssk} is
$$
\widetilde{\mathsf{N}}^{Y\bullet}_{\beta,d}
(\nu,\nu^\vee)=
\sum_{g\in \mathbb{Z}} u^{2g-2} \, 
\widetilde{\mathsf{N}}^{Y\bullet}_{g,\beta,d}
(\nu,\nu^\vee)\ .
$$
Finally, $\Fz(\nu)= |\Aut(\nu)| \prod_{i} \nu_i$ is the usual combinatorial factor. 
Formula \eqref{kssk} and Proposition 1 together imply the following result.

\begin{prop} Definition \eqref{fccf} for $\mathsf{N}^X_{g,\beta,d}$
 is {\em independent}
of the choice of $\beta^\vee$ satisfying \eqref{fwwf}.
\end{prop}



\section{The Igusa cusp form  $\chi_{10}$}

Let $\BH_2$ denote the Siegel upper half space. The
standard coordinates are
\[ \Omega =
\begin{pmatrix} \tau & z \\ z & \widetilde{\tau} \end{pmatrix} \in \BH_2\, ,
\]
where
$\tau, \widetilde{\tau} \in \BH_1$ lie in the Siegel upper half plane,
$z \in \BC$, and
$$\text{Im}(z)^2 < \text{Im}(\tau) \text{Im}(\tilde{\tau})\ .$$
We denote the exponentials of the coordinates by
\[ p = \exp(2 \pi i z), \quad q = \exp(2 \pi i \tau), \quad \tilde{q} = \exp(2 \pi i \tilde{\tau}). \]

For us, the variable $p$
is related to the genus parameter $u$ of Gromov-Witten theory
and the Euler characteristic parameter $y$ of stable pairs theory:
\[ p = \exp(i u)\, , \quad \quad y = - p\, . \]
More precisely, we have $u = 2 \pi z$ and $y = \exp(2 \pi i (z + 1/2))$.

In the partition functions,
the variable $q$ indexes classes of $S$,
$$q^{h-1} \ \ \longleftrightarrow\ \   \text{a primitive  class $\beta_h$ on $S$
satifying} \ \ 2h-2=\langle \beta_h,\beta_h\rangle\ ,$$
and the variable $\tilde{q}$ indexes classes of $E$, 
$$\tilde{q}^{d-1} \ \ \longleftrightarrow\ \   \text{$d$ times the
class $[E]$}\ .$$

We will require several special functions.
Let
\[ C_{2k}(\tau) = - \frac{B_{2k}}{2k (2k)!} E_{2k}(\tau) \]
be renormalized Eisenstein series:
\[ C_2 = - \frac{1}{24} E_2\, , \quad C_4 = \frac{1}{2880} E_4\, , \quad \ldots
\, . \]
Define the Jacobi theta  function by
\begin{align*}
F(z,\tau) & = \frac{\vartheta_1(z,\tau)}{\eta^3(\tau)} \\
 & = -i(p^{1/2} - p^{-1/2}) \prod_{m \geq 1} \frac{ (1-pq^m) (1-p^{-1}q^m)}{ (1-q^m)^2 } \\
 & = u \exp \Big( - \sum_{k \geq 1} (-1)^k C_{2k} u^{2k} \Big),
\end{align*}
where we have choosen the normalization{\footnote{From 
the point of Gromov-Witten theory, the leading term $u^k$ for 
the special functions is more natural. However, 
the usual convention in the literature is to take leading term $(2\pi i z)^k$.
We follow the usual convention for most of the classical functions.
Our convention for $F$ is an exception which 
allows for fewer signs in the statement of the Gromov-Witten and pairs results,
but results in sign changes when comparing with classical function (see Conjecture A).}}
\begin{equation}\label{vvv}
 F = u + O(u^2), \quad u = 2 \pi z\, . 
\end{equation} 
Define the Weierstrass $\wp$ function by
\begin{align*}
 \wp(z,\tau) & = - \frac{1}{u^2} - \sum_{k \geq 2} (-1)^k (2k-1) 2k C_{2k} u^{2k-2} \\
& = \frac{1}{12} + \frac{p}{(1-p)^2} + \sum_{k,r \geq 1} k (p^k - 2 + p^{-k}) q^{kr} .
\end{align*}
$F(z,\tau)$ and $\wp(z,\tau)$ are related by the following construction. Let
\begin{equation}\label{zzz}
G = F \partial_z^2(F) - \partial_z(F)^2 = F^2 \partial_z^2 \log(F)\, , 
\end{equation}
where $\partial_z = \frac{1}{2 \pi i} \frac{\partial}{\partial z} = \frac{1}{i} \frac{\partial}{\partial u} = p \frac{d}{dp}$. 
Then we have the basic relation
\begin{align} \label{fttf}
\wp(z,\tau) & = - \partial_z^2 ( \log(F(z,\tau)) ) - 2 C_2(\tau) \\
& = - \frac{G}{F^2} + \frac{1}{12} E_2\, . \nonumber
\end{align}

Define the coefficients $c(m)$ by the expansion
\[ Z(z,\tau) = -24 \wp(z,\tau) F(z,\tau)^2 = \sum_{n \geq 0} \sum_{k \in \BZ} c(4n - k^2) p^k q^n. \]
The Igusa cusp form $\chi_{10}(\Omega)$ may be expressed by a result of
Gritsenko and Nikulin \cite{GN} as
\begin{equation} \chi_{10}(\Omega) = p q \tilde{q} \prod_{(k,h,d)} ( 1 - p^k q^h \tilde{q}^d )^{c(4 h d - k^2)}, \label{101} \end{equation}
where the product is over all  $k \in \BZ$ and
$h,d \geq 0$ satisfying one of the following two conditions:
\begin{enumerate}
\item[$\bullet$]
 $h>0$ or $d>0$ ,
\item[$\bullet$]
  $h = d = 0$ and $k < 0$ . 
\end{enumerate}
It follows from \eqref{101}, that the form $\chi_{10}$ is symmetric in the variables $q$ and $\tilde{q}$,
\begin{equation}\label{hbbh}
\chi_{10}(q,\tilde{q})= \chi_{10}(\tilde{q},q)\ .
\end{equation}

Let $\phi|_{k,m} V_l$ denote the action of the $l^{\text{th}}$ Hecke operator 
on a Jacobi form $\phi$ of index $m$ and weight $k$, see 
\cite[page 41]{EZ}.
The definition \eqref{101} is equivalent to
\begin{equation} \chi_{10}(\Omega) = - \tilde{q}\cdot F(z,\tau)^2 \Delta(\tau) \cdot \exp\Big( - \sum_{l \geq 1} \tilde{q}^l \cdot (Z|_{0,1}V_l)(z,\tau) \Big)\, , \label{102} \end{equation}
where
$$\Delta(\tau)= {q \prod_{n>0} (1-q^n)^{24}}\ .$$
Alternatively, we may define $\chi_{10}(\Omega)$ as the additive lift,
\[ \chi_{10}(\Omega) = - \sum_{\ell \geq 1} \tilde{q}^\ell \cdot \big( F^2 \Delta \big|_{10,1} V_{\ell} \big)(z,\tau) \,. \]

Our main interest is in the inverse of the Igusa cusp form,
\[ \frac{1}{\chi_{10}(\Omega)} \,. \]
By \eqref{vvv} and \eqref{102}, $\frac{1}{\chi_{10}}$ has a pole of order $2$ at $z = 0$ and its translates. Hence, the Fourier expansion of 
$\frac{1}{\chi_{10}}$ depends on the location in $\Omega$. We will always assume the parameters $(z,\tau)$ to be in the region
\[ 0 < |q| < |p| < 1. \]
The above choice determines the Fourier expansion of $\frac{1}{F^2 \Delta}$ and therefore also of $\frac{1}{\chi_{10}}$.

Consider the expansion in $\tilde{q}$,
$$ \frac{1}{\chi_{10}(\Omega)} = \sum_{n \geq -1} \tilde{q}^n \psi_n \, .$$
For the first few terms (see \cite[page 27]{KK3}), we have
\begin{align*}
 \psi_{-1} & = - \frac{1}{F^2 \Delta} \\
 \psi_{0}\, \, \,  & = 24 \frac{\wp}{\Delta} \\
 \psi_{1}\, \, \,  & = -\left(324 \wp^2 + \frac{3}{4} E_4\right) \frac{F^2}{\Delta} \\
 \psi_{2}\, \, \,  & = \left(3200 \wp^3 + \frac{64}{3} E_4 \wp + \frac{10}{27} E_6\right) \frac{F^4}{\Delta} \, .
\end{align*}
In particular, the leading coefficient (with $p = -y$) is
\begin{align*}
\psi_{-1}
& = \frac{-1}{y + 2 + y^{-1}}\, \frac{1}{q} \prod_{m \geq 1} \frac{1}{(1 + y^{-1} q^m)^2 (1 - q^m)^{20} (1 + y q^m)^2}\ .
\end{align*}
It is related to
the Katz-Klemm-Vafa formula for $K3$ surfaces proven in \cite{MPT,PT2},
\begin{align*} 
-\psi_{-1}\ \Big|_{-y=\exp(-iu)} & = 
\sum_{\substack{h \geq 0 \\ g \geq 0}} u^{2g-2} q^{h-1} \int_{\M_g(S,\beta_h)} (-1)^g \lambda_g \\ 
& = \frac{1}{u^2 \Delta(\tau)} \exp\Big( \sum_{k \geq 1} u^{2k} \frac{|B_{2k}|}{k \cdot (2k)!} E_{2k}(\tau) \Big) \ .
\end{align*}
The functions $\psi_d$ are meromorphic Jacobi forms with poles 
of order $2$ at $z = 0$ and its translates. The principal part of $\psi_d$ at $z = 0$ equals
\begin{equation} \frac{a(d)}{\Delta(\tau)} \frac{1}{(2 \pi i z)^2}\ 
\label{aaaa}  
\end{equation}
where $a(d)$ is the $q^d$ coefficient of $\frac{1}{\Delta}$.

%
%
%

\section{Hilbert schemes of points}

\subsection{Curves classes}
Let $S$ be a nonsingular projective $K3$ surface. Let 
$$S^{[d]}=\text{Hilb}^d(S)$$
denote the 
the Hilbert scheme of $d$ points of $S$. The Hilbert scheme
$S^{[d]}$ is a nonsingular projective variety of dimension
$2d$. Moreover, $S^{[d]}$ carries a holomorphic symplectic form,
see  \cite{Beau, Nak}.

We follow standard notation for the Nakajima operators \cite{Nak}.
For $\alpha \in H^{\ast}(S ; \BQ)$ and $i > 0$, let 
\[ \Fp_{-i}(\alpha) : H^{\ast}(S^{[d]},\BQ) \ra H^{\ast}(S^{[d+i]},\BQ),\ \ \ \gamma \mapsto \Fp_{-i}(\alpha) \gamma \]
be the Nakajima creation operator defined by adding length $i$
punctual subschemes incident to a cycle Poincare dual to $\alpha$.  
The cohomology of $S^{[d]}$ can be completely described  
by the cohomology of $S$ via the action of the operators $\Fp_{-i}(\alpha)$ on the vacuum vector 
\[ 1_S \in H^{\ast}(S^{[0]},\BQ) = \BQ. \]

Let $\pt$ be the class of a point on $S$. 
For $\beta \in H_2(S,\BZ)$, define the class
\[ C(\beta) = \Fp_{-1}(\beta) \Fp_{-1}(\pt)^{d-1} 1_S \in H_2(S^{[d]},\BZ). \]
If $\beta = [C]$ for a curve $C \subset S$, then $C(\beta)$ is the class of the curve given by fixing $d-1$
distinct points away from $C$ and letting a single point move on $C$. For $d \geq 2$, let
\[ A = \Fp_{-2}(\pt) \Fp_{-1}(\pt)^{d-2} 1_S \]
be the class of an exceptional curve -- the locus of spinning
double points centered at a point $s \in S$ plus $d-2$ fixed points away from 
$s$.
For $d \geq 2$, 
$$ H_2(S^{[d]},\BZ) =
\big\{ \ C(\beta) + kA \ \big| \  
\beta \in H_2(S,\BZ), \  k \in \BZ\ \big\} \ .$$

The moduli space of stable maps{\footnote{Here, the maps are required
to have {\em connected} domains. No superscript $\bullet$ appears
in the notation.}}  
$\Mbar_{0,2}(S^{[d]}, C(\beta) + kA)$
carries a reduced virtual class of 
dimension $2d$.

\subsection{Elliptic fibration}
\label{hvvc}
Let $S$ be an elliptic $K3$ surface 
$$\pi : S \ra \p^1$$ 
with a section, and let $F \in H_2(S,\BZ)$ be the class of a fiber.  
The generic fiber of the induced fibration
\[ \pi^{[d]} : \Hilb^d(S) \ra \Hilb^d(\p^1) = \p^d, \]
is a nonsingular Lagrangian torus.  
Let 
$$L_z  \subset \Hilb^d(S)$$ 
denote the the fiber of $\pi^{[d]}$ over $z\in \p^d$.

Let $\beta_h$ be a primitive curve class on $S$ with $\langle \beta_h,F 
\rangle = 1$ and square 
$$\langle \beta_h, \beta_h \rangle=2h-2\, .$$
 For $z_1, z_2 \in \p^d$, define the invariant
\begin{eqnarray*}
  \mathsf{N}^{\Hilb}_{k,h,d} &=& \blangle  L_{z_1}, L_{z_2} \brangle_{\beta_h,k}^{\Hilb^d(S)} \\
&= & 
\int_{[\overline{M}_{0,2}(S^{[d]}, C(\beta_h) + kA)]^{\text{red}}} \ev_1^{\ast}(L_{z_1}) 
\cup  \ev_2^{\ast}(L_{z_2})
\end{eqnarray*}
which (virtually) counts the number of rational curves incident to the Lagrangians $L_{z_1}$ and $L_{z_2}$.

A central result of \cite{GO} is the following complete evaluation of 
$\mathsf{N}^{\Hilb}_{k,h,d}$. 
\begin{thm} \label{MThm0} 
For $d>0$, we have
\begin{equation*}
\sum_{k \in \BZ} \sum_{h\geq 0}
\mathsf{N}^{\Hilb}_{k,h,d}\, y^k q^{h-1}  = \frac{F(z,\tau)^{2d-2}}{\Delta(\tau)}
\end{equation*}
where $y=-e^{2\pi i z}$ and $q=e^{2\pi i\tau}$. 
\end{thm}

In the $d=1$ case, the class $A$ vanishes on $S^{[1]}=S$.
By convention, only the $k=0$ term in the sum on the left is taken. 
Then, Theorem \ref{MThm0} specializes in $d=1$ to
the Yau-Zaslow formula \cite{YZ} for rational curve counts in
primitive classes of $K3$ surfaces.

If we specialize the formula of Theorem \ref{MThm0} to $d=0$, we obtain
\begin{equation*}
\sum_{k \in \BZ} \sum_{h\geq 0}\mathsf{N}^{\Hilb}_{k,h,0}\, y^kq^{h-1}  = \frac{F(z,\tau)^{-2}}{\Delta(\tau)} = \frac{1}{F(z,\tau)^2 \Delta}\ .
\end{equation*}
The result is exactly the Katz-Klemm-Vafa formula as discussed in
Section 2. While the $d=0$ specialization is not
geometrically well-defined from the point of view of the
Hilbert scheme, the result strongly suggests a correspondence
between the Gromov-Witten theory of $K3$ fibrations and
the reduced theory of $\Hilb^d(S)$. Precise conjectures
will be formulated in the next Section.

\section{Conjectures}

\subsection{Primitive case}
Let $\beta_h\in \text{Pic}(S) \subset H_2(S,\mathbb{Z})$ be a {\em primitive}
class which 
is positive (with respect to any ample polarization) and
satsifies
$$\langle \beta_h, \beta_h \rangle = 2h-2\, .$$
Let $(E,0)$ be a nonsingular elliptic curve with origin $0\in E$.
For $d>0$, consider  the reduced Gromov-Witten invariant
\begin{equation} \label{wssw}
 \mathsf{H}_d(y,q) = 
\sum_{k \in \BZ} \sum_{h\geq 0}y^k q^{h-1} 
\int_{[ \overline{M}_{(E,0)}(S^{[d]}, C(\beta_h) + kA) ]^{\text{red}}}  \text{ev}_0^*(\beta_{h,k}^\vee)\, .
\end{equation}
The moduli space \eqref{wssw} is of stable maps with 1-pointed domains
with complex
structure {\em fixed} after stabilization
to be $(E,0)$.
The reduced virtual dimension of 
$\overline{M}_{(E,0)}(S^{[d]}, C(\beta_h) + kA)$ is 1.
The divisor class $\beta_{h,k}^\vee \in H^2(S^{[d]},\mathbb{Q})$
is chosen to satisfy
\begin{equation}\label{ffqq}
\int_{C(\beta_h) + kA} \beta_{h,k}^\vee  = 1 \, . 
\end{equation}
The integral \eqref{wssw} is well-defined.

Following the perspective of \cite{BP,OPH,OPLC},
a connection between the disconnected Gromov-Witten 
invariants $\mathsf{N}^\bullet_{g,\beta_h,d}$ of $K3\times E$ and the series \eqref{wssw}
obtained from the geometry of $S^{[d]}$ is natural to expect.

We may rewrite $\mathsf{H}_d(y,q)$ by degenerating $(E,0)$ to the
nodal elliptic curve (and using the divisor equation) as
\begin{equation} \label{wsswss}
 \mathsf{H}_d(y,q) = 
\sum_{k \in \BZ} \sum_{h\geq 0}
y^k q^{h-1} 
\int_{[ \overline{M}_{0,2}(S^{[d]}, C(\beta_h) + kA) ]^{\text{red}}}  
(\ev_1 \times \ev_2)^{\ast} [\Delta^{[d]}]\, ,
\end{equation}
where $[\Delta^{[d]}] \in H^{2d}(S^{[d]} \times S^{[d]},\mathbb{Q})$ is the diagonal class.
Equation \eqref{wsswss} shows the integral \eqref{wssw} is
independent of the choice of $\beta_{h,k}^\vee$ satisfying
\eqref{ffqq}. By convention,
\begin{eqnarray*} \label{wsswssw}
 \mathsf{H}_1(q) & = &  
 \sum_{h\geq 0}
 q^{h-1} 
\int_{[ \overline{M}_{0,2}(S^{[1]}, C(\beta_h)) ]^{\text{red}}}  
(\ev_1 \times \ev_2)^{\ast} [\Delta^{[1]}] \\ \nonumber
& = & 2 q \frac{d}{dq}\Delta^{-1}\,\\  \nonumber 
& = &\, -2 \frac{E_2}{\Delta}\ .
\end{eqnarray*}
For the second equality, we have used the Yau-Zaslow
formula.

We define a generating series over all $d>0$ of the Hilbert
scheme geometry:
\[ \mathsf{H}(y, q, \tilde{q}) = 
\sum_{d > 0} \mathsf{H}_d(y,q)\, \tilde{q}^{d-1}. \]
The analogous generating series over all $d$ for the 3-fold geometry
$$X=S\times E$$
is defined by
\begin{equation}\label{fppf}
 \mathsf{N}^{X\bullet}(u, q, \tilde{q}) = 
\sum_{g \in \mathbb{Z}}\sum_{h\geq 0} \sum_{d\geq 0}
 \mathsf{N}_{g,\beta_h,d}^{X\bullet} \ u^{2g-2} q^{h-1}\tilde{q}^{d-1}. 
\end{equation}
The main conjecture in the primitive case is the following.

\vspace{8pt}
\noindent{\bf Conjecture A.}
{\em Under $y = -\exp(iu)$,
$$\mathsf{N}^{X\bullet}(u,q,\tilde{q}) \ = \ 
 \mathsf{H}(y, q, \tilde{q})  + \frac{1}{F^2\Delta} \cdot
\frac{1}{\tilde{q}} \prod_{n \geq 1} \frac{1}{(1 - (\tilde{q} \cdot G)^n)^{24}} 
\ = \ - \frac{1}{\chi_{10}(\Omega)}\ .$$}
\vspace{8pt}

The Igusa cusp form $\chi_{10}(\Omega)$ and the
functions $F(z,\tau)$, $\Delta(\tau)$, and $G(z,\tau)$ 
are as defined in Section 2.

The second factor
in the correction term added to $\mathsf{H}(y, q, \tilde{q})$ can be expanded as
\begin{align*} \frac{1}{\widetilde{q}} \prod_{n \geq 1} \frac{1}{(1 - (\tilde{q} \cdot G)^n)^{24}} & = G \cdot \frac{1}{\Delta(\tilde{\tau})}\Big|_{\tilde{q} = G \cdot \tilde{q}} \\
 & = \tilde{q}^{-1} + 24 G + 324 G^2 \tilde{q} + 3200 G^3 \tilde{q}^2 + 
\cdots\, .
\end{align*}
From definition \eqref{zzz} of $G$ and property \eqref{vvv},
 $$G = 1 + O(q)\, .$$ 
Hence the full correction term has $\tilde{q}^{-1}$ coefficient
$\frac{1}{F^2\Delta}$
which is the Katz-Klemm-Vafa formula (required since
$\mathsf{H}(y, q, \tilde{q})$ has no $\tilde{q}^{-1}$ term).
The $\tilde{q}^0$ term yields the identity
$$-2 \frac{E_2}{\Delta} + 24\frac{G}{F^2\Delta}  = 
-24 \frac{\wp}{\Delta} \ $$
which is equivalent to \eqref{fttf}.

We do not at present have a geometric explanation for the
full correction term
\begin{equation} \label{cxxc} \frac{1}{F^2\Delta} \cdot 
\frac{1}{\tilde{q}} 
\prod_{n \geq 1} \frac{1}{(1 - (\tilde{q} \cdot G)^n)^{24}} \ .
\end{equation}
Denote the $\tilde{q}^d$ coefficient of \eqref{cxxc} by 
\[ \phi_d = \frac{a(d)}{\Delta(\tau)} \cdot \frac{G^{d+1}}{F^2}\ . \]
Here, $a(d)$ is the $q^d$ coefficient of $\frac{1}{\Delta}$. 
Then $\phi_d$ is a meromorphic Jacobi form with poles of order 2
 at $z = 0$ and its translates. The principal part of $\phi_d$ at $z=0$ 
equals
\[ \frac{a(d)}{\Delta(\tau)} \frac{1}{(2 \pi z)^2}. \]
Comparing with \eqref{aaaa}, we see $\phi_d$ accounts for 
all the poles in $-\psi_d$. The second equality in Conjecture A 
therefore determines a \emph{natural} splitting
\begin{equation} -\psi_d = \mathsf{H}_d + \phi_d \label{123} \end{equation}
of $-\psi_d$ into a finite (holomorphic) quasi-Jacobi form $\mathsf{H}_d$ 
and a polar part $\phi_d$. 
In particular, the Fourier expansion of $\mathsf{H}_d$ 
is independent of the moduli $\tau$. Hence, 
all  wall-crossings are related to $\phi_d$. 
The splitting of $\psi_d$ into a finite and polar part has been studied by Dabholkar, Murthy, and Zagier \cite{DMZ} and 
has a direct interpretation in a physical model of quantum black holes.
In fact, up to the $E_2$ summand in \eqref{fttf}
our splitting matches their simplest choice, see \cite[Equations 1.5 and 9.1]{DMZ}.

The two equalities of Conjecture A are independent claims.
The first is a correspondence result (up to correction).
We have made verifications by partially evaluating both sides.
The second equality, which evalutes the series,
has already been seen to hold for the coefficients of $\tilde{q}^{-1}$ and $\tilde{q}^0$.
The second equality has been proven for the coefficient of $\tilde{q}^{1}$  
in \cite{GO}.
The conjecture
\begin{equation}\label{zxxz}
\mathsf{N}^{X\bullet}(u,q,\tilde{q}) \ = 
\  - \frac{1}{\chi_{10}(\Omega)}\ 
\end{equation}
is directly related to the predictions of Section 6.2 of
\cite{KKV}. J. Bryan \cite{jb}  has verified{\footnote{Bryan's calculation is
on the sheaf theory side, see Conjecture D below.}}
conjecture \eqref{zxxz} for the coefficients
$q^{-1}$ and $q^{0}$.

The conjectural equality  \eqref{zxxz} may be viewed as a
mathematically precise formulation of 
\cite[Section 6.2]{KKV}. The Igusa cusp form $\chi_{10}$ appears 
in \cite{KKV} via the elliptic genera of the symmetric products
of a $K3$ surface (the $\chi_{10}$ terminology is not used in \cite{KKV}). The 
development of the reduced virtual class
occurred in the years following \cite{KKV}.
Since the $K3 \times E$ geometry carries a free $E$-action,
a further step (beyond reduction) must be taken to avoid
a trivial theory. Definition \eqref{fccf} with an insertion
is a straightforward solution. Finally, the Igusa cusp form
$\chi_{10}$ is related to the {\em disconnected} reduced Gromov-Witten
theory of $K3 \times E$. With these foundations,
the prediction of
\cite{KKV} may be interpreted to exactly match \eqref{zxxz}.

By the symmetry \eqref{hbbh} of the Igusa cusp form $\chi_{10}$,
Conjecture A predicts a surprising symmetry for
disconnected Gromov-Witten theory of $X$,
$$\mathsf{N}^{X\bullet}_{g,\beta_h,d} = \mathsf{N}^{X\bullet}_{g,\beta_d,h}\, ,$$
for all primitive classes $\beta_h$ and $\beta_d$. In the notation
\eqref{vvttt}, the symmetry can be written as 
$$\mathsf{N}^{X\bullet}_{g,1,h,d} = \mathsf{N}^{X\bullet}_{g,1,d,h}\, $$
for all $h,d\geq0$
(where the subscript $1$ denotes primitivity).

Conjectures for the motivic generalization of the $d=0$ case are
presented in \cite{KKP}. An interesting connection to the
Mathieu $\mathsf{M}_{24}$ moonshine phenomena appears there in the data.
Since the Gromov-Witten theory of $X$ is related via 
$-\frac{1}{\chi_{10}}$ by Conjecture A to 
the elliptic
genera of the symmetric products of $K3$ surfaces,
the Mathieu $\mathsf{M}_{24}$ moonshine must also arise here.

\subsection{Imprimitive classes}
The generating series $\mathsf{N}^{X\bullet}(u, q, \tilde{q})$
defined by \eqref{fppf} concerns only the primitive classes 
$\beta_h\in \text{Pic}(S)$.
To study the imprimitive case, we define
\begin{equation}\label{ttllq}
\mathsf{N}^{X}_{\beta}(u, \tilde{q}) = 
\sum_{g \in \mathbb{Z}} \sum_{d\geq 0}
 \mathsf{N}^X_{g,\beta,d} \ u^{2g-2} \tilde{q}^{d-1} 
\end{equation}
for any $\beta\in \text{Pic}(S)$. The
coefficents of $ \mathsf{N}^{X}_{\beta}(u, \tilde{q})$ are
{\em connected} invariants.{\footnote{By Proposition 1, there
is no difficulty in moving back and forth between connected
and disconnected invariants.}}
We may write \eqref{ttllq} in the notation \eqref{vvttt} as
\begin{equation*}
\mathsf{N}^{X}_{m\beta_h}(u, \tilde{q}) = 
\sum_{g \geq 0} \sum_{d\geq 0}
 \mathsf{N}^X_{g,m,m^2(h-1)+1,d} \ u^{2g-2} \tilde{q}^{d-1} 
\end{equation*}
for primitive $\beta_h\in \text{Pic}(S)$ satisfying
$$\langle \beta_h, \beta_h \rangle = 2h-2\, .$$
In the primitive ($m$=1) case, instead of writing $\mathsf{N}^{X}_{\beta_h}$,
we write
\begin{equation*}
\mathsf{N}^{X}_{h}(u, \tilde{q}) = 
\sum_{g \geq 0} \sum_{d\geq 0}
 \mathsf{N}^X_{g,1,h,d} \ u^{2g-2} \tilde{q}^{d-1} 
\end{equation*}
\vspace{10pt}
\noindent{\bf{Conjecture B.}}
{\em For all $m>0$,
\begin{equation}\label{xxrr}
\mathsf{N}^{X}_{m\beta_h}(u, \tilde{q}) = \sum_{k|m}
\frac{1}{k}\mathsf{N}^{X}_{{(\frac{m}{k})^2(h-1)+1} }
(ku, \tilde{q})\, ,
\end{equation}
for the primitive class $\beta_h$.}
\vspace{10pt}

Conjecture B expresses the series $\mathsf{N}^{X}_{m\beta_h}$
in terms of series for {\em primitive classes} corresponding
to the divisors $k$ of $m$.
To such a divisor $k$, we associate the class $\frac{m}{k}\beta_h$ with square
$$\left\langle \frac{m}{k} \beta_h, \frac{m}{k} \beta_h \right\rangle = 
\left(\frac{m}{k}\right)^2(2h-2) = 2\left({{\left(\frac{m}{k}\right)^2(h-1)+1} }\right)-2\ .$$
The term in the sum on the right side of \eqref{xxrr} corresponding
to $k$ may be viewed as the contribution of the primitive class
of square equal to $\langle \frac{m}{k} \beta_h, \frac{m}{k} \beta_h \rangle$.
The primitive contribution of the divisor $k$ 
to $\mathsf{N}^X_{g,m,m^2(h-1)+1,d}$
is 
\begin{equation}
\label{gvvp}
k^{2g-3} \cdot \mathsf{N}^X_{g,1,\left(\frac{m}{k}\right)^2(h-1)+1,d}\ .
\end{equation}

The scaling factor $k^{2g-3}$
is {\em independent} of $d$. In fact, the variable $\tilde{q}$
 plays no role in formula \eqref{xxrr}. To emphasize the point,
the contribution of the divisor $k$ geometrically is a
contribution of the class 
$$\left(\frac{m}{k}\beta_h,d\right) = \iota_{S*}\left(\frac{m}{k}\beta\right)+
\iota_{E*}(d[E])$$ to
$(m\beta_h,d)$ in the 3-fold $S\times E$.
Unless $d=0$, such a contribution can not be viewed as
a multiple cover contribution in the usual Gopakumar-Vafa
perspective of Calabi-Yau 3-fold invariants.

In the $d=0$ case, Conjecture B specializes to the multiple
cover structure of the KKV conjecture proven in \cite{PT2}
which is usually formulated in terms of BPS counts.
We could rewrite Conjecture B in terms of nonstandard 3-fold BPS counts
which do not interact with the curve class $[E]$ associated to 
$\tilde{q}$. Instead, we have written Conjecture B in the most
straightforward Gromov-Witten form. In fact, the simple
form of Conjecture B suggests a much more general underlying structure
for $K3$ surfaces (which we will discuss in Section \ref{dd44}).

Further evidence for Conjecture B can be found in case $h=0$.
Localization arguments{\footnote{The localization
required here is parallel to the proof of the
scaling in \cite[Theorem 3]{FP}.}}
 (with respect
to the $\mathbb{C}^*$ acting on the $-2$ curve) 
yield 
\begin{equation*}
\mathsf{N}^{X}_{m\beta_0}(u, \tilde{q}) = 
\frac{1}{m}\mathsf{N}^{X}_{0}
(mu, \tilde{q})\, .
\end{equation*}
Hence, Conjecture B predicts the primitive contributions
corresponding to $k\neq m$ all {\em vanish} in the $h=0$ case.
Such vanishing is correct: the reduced Gromov-Witten invariants of $X$ vanish
for classes $(\beta,d)$ where $\beta$ is primitive and 
$$\langle \beta,\beta \rangle <-2\ .$$

Finally, an elementary analysis leads to the proof of Conjecture B
in all cases for $g=1$. Both sides of \eqref{xxrr} are easily
calculated.

\subsection{Descendent theory for $K3$ surfaces} \label{dd44}
Let $S$ be a nonsingular projective $K3$ surface, and let  
$\beta\in \text{Pic}(S)$ be a positive class.
We define the (reduced) descendent Gromov-Witten invariants by
$$\left\langle \prod_{i=1}^n \tau_{\alpha_i}(\gamma_i) \right\rangle_{g,\beta}^S=
\int_{[\overline{M}_{g,n}(S,\beta)]^{red}} \prod_{i=1}^n \psi_i^{\alpha_i}
\cup \text{ev}_i^\ast(\gamma_i)\ , \ \ \ \gamma_i\in H^*(S,\mathbb{Q})\ .$$
A potential function for the
 descendent theory of $K3$ surfaces in primitive classes is  defined
by
\begin{equation}
\label{vppqq}
\mathsf{F}_{g}\big(\tau_{k_1}(\gamma_{l_1}) \cdots
\tau_{k_r}(\gamma_{l_r})\big)=
\sum_{h=0}^\infty 
\Big\langle \tau_{k_1}(\gamma_{l_1}) \cdots
\tau_{k_r}(\gamma_{l_r})\Big\rangle^{S}_{g,\beta_h} 
\ q^{h-1}
\end{equation}
for $g\geq 0$.

The 
descendent potential \eqref{vppqq} is 
a quasimodular form \cite{MPT}. 
The ring 
$$\QMod = \QQ[E_{2}(q), E_4(q), E_6(q)]$$
of holomorphic quasimodular forms (of level $1$) is the 
$\QQ$-algebra generated by Eisenstein series $E_{2k}$, see \cite{bghz}.  
The ring $\QMod$
is naturally graded by weight (where $E_{2k}$ has weight $2k$) 
and inherits an increasing filtration 
$$\QMod_{\leq 2k}\subset \QMod$$
given by forms of weight $\leq 2k$.
The precise
result proven in \cite{MPT} is the following.

\begin{thm} \label{qqq} The descendent potential
is the Fourier expansion in $q$
of a quasimodular form 
$$\mathsf{F}_{g}(\tau_{k_{1}}(\gamma_{{1}})\cdots\tau_{k_{r}}
(\gamma_{{r}}))
\in \frac{1}{\Delta(q)}\,\QMod_{\leq 2g+2r}\,$$
with pole at $q=0$ of
order at most $1$.
\end{thm}


\noindent Conjectures C1 and C2 below will reduce all descendent invariants to
the primitive case.

Conjecture C1  is an invariance property. Let $S$ and $\widetilde{S}$
be two $K3$ surfaces, and
let
$$\varphi:\Big(H^2(S,\mathbb{R})\, ,\langle,\rangle\Big) 
\rightarrow \Big(H^2(\widetilde{S}\, ,\mathbb{R}),\langle,\rangle \Big)$$
be a {\em real isometry} sending
a effective curve{\footnote{Since there is a canonical isomorphism
$H_2(S,\mathbb{Z})\stackrel{\sim}{=}H^2(S,\mathbb{Z})$ ,
we may consider $\beta$ also as a cohomology class.}} class $\beta\in H^2(S,\mathbb{Z})$ to an
effective curve class $\widetilde{\beta}\in H^2(\widetilde{S},\mathbb{Z})$,
$$\varphi(\beta)= \widetilde{\beta}\ .$$
It is convenient to extend $\varphi$ to all of $H^*(S,\mathbb{R})$ by
$$\varphi(\mathsf{1})=\mathsf{1}\, , \ \ \ \ \varphi(\mathsf{p})=\mathsf{p}\, $$
where $\mathsf{1}$ and $\mathsf{p}$ are the identity and point
classes respectively.

\vspace{10pt}
\noindent{\bf{Conjecture C1.}}
If $\beta\in H^2(S,\mathbb{Z})$ and $\widetilde{\beta}\in H^2(S,\mathbb{Z})$
have the same divisibility, 
\begin{equation*}
\left\langle \prod_{i=1}^r \tau_{\alpha_i}(\gamma_i) \right\rangle_{g,\beta}^S
= 
\left\langle \prod_{i=1}^n \tau_{\alpha_i}(\varphi(\gamma_i)) 
\right\rangle_{g,\widetilde{\beta}}^{\widetilde{S}}\ .
\end{equation*}
\vspace{10pt} 

Let $\delta_i$ be the (complex) codimension of $\gamma_i$,
$$\gamma_i \in H^{2\delta_i}(S,\mathbb{Q})\ .$$
Conjecture C1 implies the invariant $\big\langle \prod_{i=1}^r \tau_{\alpha_i}(\gamma_i) \big\rangle_{g,\beta}^S$ depends only
upon $g$, the divisibility of $\beta$, and all the
pairings
$$\langle \gamma_i, \gamma_j \rangle\,, \ \ 
\langle \gamma_i, \beta \rangle\,, \ \
\langle \beta, \beta \rangle$$
for $\delta_i=\delta_j=1$.
For the Gromov-Witten theory of curves, a similar invariance 
statement
has been proven in \cite{OPvir}.

Conjecture C2 expresses descendent invariants in imprimitive
classes in term of primitive classes.
Let $\beta_{h}$ be a primitive curve class on $S$.
Since all invariants vanish if $h<0$, we assume $h\geq 0$. 
Let $m$ be a positive integer.
For every divisor $k$ of $m$, let $S_k$ be a $K3$ surface
with a real isometry
$$\varphi_k:\Big(H^2(S,\mathbb{R})\, ,\langle,\rangle\Big) 
\rightarrow \Big(H^2(S^k\, ,\mathbb{R}),\langle,\rangle \Big)$$
for which
$\varphi(\frac{m}{k}\beta_h)$ is a primitive and effective
curve class on $S_k$. 
\begin{enumerate}
\item[$\bullet$]
If $h>0$, such $S_k$ are easily found. 
\item[$\bullet$] If $h=0$, such $S_k$
exist only in the
$k=m$ case.
\end{enumerate}

\vspace{10pt}
\noindent{\bf{Conjecture C2.}} 
{\em For primitive classes $\beta_h$ and  $m>0$,
\begin{equation*}
\left\langle \prod_{i=1}^r \tau_{\alpha_i}(\gamma_i) \right\rangle_{g,m\beta_h}^S
= \ \sum_{k|m}
k^{2g-3+\sum_{i=1}^n \delta_i}
\left\langle \prod_{i=1}^n \tau_{\alpha_i}(\varphi_{k}(\gamma_i)) 
\right\rangle_{g,\varphi_k(\frac{m}{k} \beta_h)}^{S_k}\ .
\end{equation*}}
\vspace{10pt} 

In the $h=0$ case, the $k\neq m$ terms on the
right side of the equality in Conjecture C2 are
defined to vanish. By Conjecture C1, the right side is
independent of the choices of $S_k$ and $\varphi_k$.

The first evidence: the KKV formula interpreted 
as the Hodge integral \eqref{kkv}
exactly satisfies Conjecture C2 with the integrand viewed
as having {\em no} descendent insertions. In fact, 
$(-1)^g\lambda_g$ can be expanded in terms of descendent
integrands on strata --- applying Conjecture C2 to
such an expansion exactly yields the multiple
cover scaling of the KKV formula.
In particular, Conjecture C2 together with the KKV formula
in the primitive case {\em implies} the full KKV formula.

Conjecture B, when fully expanded, has a scaling factor of 
$k^{2g-3}$ which corresponds to Conjecture C2 with no
insertions. In fact, Conjecture B {\em follows} from Conjecture
C2 via the product formula \cite{Beh} for virtual
classes in Gromov-Witten theory. Conjecture C2 was motived
for us by Conjecture B.

The second evidence: Maulik in \cite[Theorem 1.1]{DM} calculated 
descendents for the $A_1$ singularity. We may interpret
the calculation of \cite{DM} as verifiying Conjecture C2 in
case $h=0$. The scaling of Conjecture C2 appears in \cite[Theorem 1.1]{DM}
as the final result because the primitive contributions
corresponding to $k\neq m$ all {\em vanish} in the $h=0$ case.
Of course, the $A_1$ singularity only captures codimensions $0$ and $1$ for $\delta$.

A simple example not covered by the two above cases is the
integral
\begin{equation}\label{tvvy} 
\big\langle \tau_{0}(\mathsf{p}) \big\rangle_{1,m\beta_1}^S
\end{equation}
where $\mathsf{p}\in H^4(S,\mathbb{Q})$ is the point class.
The primitive class $\beta_1$ may be taken to be the fiber
$F$ of an elliptically fibered $K3$ surface
$$\pi: S \rightarrow \mathbb{P}^1\ .$$
The primitive invariant is immediate: 
$$\big\langle \tau_{0}(\mathsf{p}) \big\rangle_{1,\beta_1}^{K3} =1 $$
Hence, Conjecture C2 yields the prediction
\begin{eqnarray*}
\left\langle \tau_{0}(\mathsf{p}) \right\rangle_{1,m\beta_h}^S
& = & \sum_{k|m}
k^{2-3+2}
\left\langle \tau_{0}(\mathsf{p}) 
\right\rangle_{1,\beta_1}^{K3}\ \\
& = & \sum_{k|m} k\, .
\end{eqnarray*}
We can evaluate \eqref{tvvy} directly from the geometry of
stable maps in the class $mF$ of $S$. The integral equals the number of connected
degree $m$ covers of an elliptic curve by an elliptic curve (times $m$ for the
insertion),
$$m \sum_{k|m} \frac{1}{k} = \sum_{k|m} k\ ,$$
which agrees with the prediction.

A much more interesting example is the genus 2 invariant
\begin{equation*} 
\big\langle \tau_{0}(\mathsf{p}),\tau_0(\mathsf{p}) \big\rangle_{2,2\beta_2}^S
\end{equation*}
in twice the primitive class $\beta_2$. Via standard geometry,
$\beta_2$ may be taken to be the hyperplane section of 
a $K3$ surface $S$ with a degree 2 cover
$$\epsilon:S \rightarrow \mathbb{P}^2$$
branched along a nonsingular sextic $$C_6 \subset \mathbb{P}^2\ .$$
Conjecture C2 predicts the following equation:
\begin{eqnarray*} 
\big\langle \tau_{0}(\mathsf{p}),\tau_0(\mathsf{p}) \big\rangle_{2,2\beta_2}^S & = & 
\big\langle \tau_{0}(\mathsf{p}),\tau_0(\mathsf{p}) \big\rangle_{2,\beta_5}^{K3}\, +\, 2^{2\cdot2-3+4} \big\langle \tau_{0}(\mathsf{p}),\tau_0(\mathsf{p}) \big\rangle_{2,\beta_2}^{K3}\, .
\end{eqnarray*}
The primitive counts can be found in \cite[Theorem 1.1]{BL},
$$ \big\langle \tau_{0}(\mathsf{p}),\tau_0(\mathsf{p}) \big\rangle_{2,\beta_2}^{K3}= 1, \ \ \
\big\langle \tau_{0}(\mathsf{p}),\tau_0(\mathsf{p}) \big\rangle_{2,\beta_5}^{K3}= 8728
\, .$$
So we obtain the prediction
\begin{equation}\label{vqq2} 
\big\langle \tau_{0}(\mathsf{p}),\tau_0(\mathsf{p}) \big\rangle_{2,2\beta_2}^S \ = \ 8728+2^5\cdot 1\ =\ 8760\, .
\end{equation}

The verification of \eqref{vqq2} is more subtle than the primitive calculation.
We study the geometry of curves in class $2\beta_2$ on the branched $K3$ surface $S$.
The two point insertions on $S$ determine two points $p,q\in \mathbb{P}^2$.
There are 3 contributions to the invariant \eqref{vqq2}:
\begin{enumerate}
\item[(i)] genus 2 curves in the series $2\beta_2$ arising as $\epsilon^{-1}(Q)$
where $Q\subset \mathbb{P}^2$ is a conic passing through $p$ and $q$ and tangent to
the branch divisor $C_6$ at 3 distinct points,
\item[(ii)] genus 2 curves which are the union of two genus 1 curves
arising as $\epsilon$ inverse images of a tangent line of $C_6$ through $p$ and
a tangent line of $C_6$ through $q$,
\item[(iii)] genus 2 curves which are the union of genus 2 and genus 0 curves
arising as the $\epsilon$ inverse
images of the unique line passing through $p$ and $q$ and 
a bitangent line of $C_6$.
\end{enumerate}
The most difficult count of the three is the first. An analysis shows there are no
excess issues, hence (i) is equal to the corresponding genus 0 relative invariant of
$\mathbb{P}^2/C_6$,
\begin{equation}\label{vzz4}
\int_{[\overline{M}_{0,2}(\mathbb{P}^2/C_6, 2)_{(1)^6 (2)^3}]^{vir}} \text{ev}_1^{-1}(p) \cup
\text{ev}_2^{-1}(q) \ =  6312\ ,
\end{equation}
where $(1)^6(2)^3$ indicates the (unordered) relative boundary condition of 3-fold 
tangency.\footnote{To calculate the relative invariant \eqref{vzz4},
we have used the program GROWI written by A. Gathmann and available 
on his webpage \cite{growi} at TU Kaiserslautern. The submission line to GROWI is
$${\text{growi}\ \ N=1,G=0,D=2,E=6,\, H^2:2,\, [1,2]:3 }\, ,$$
and the output is $37872=3!\cdot 6312$. Since GROWI orders
the 3 relative tangency points (which we do not do in \eqref{vzz4}), a
division by $3!$ completes the calculation.

In addition to providing the software, Gathmann inspired our entire
approach to $\big\langle \tau_{0}(\mathsf{p}),\tau_0(\mathsf{p}) \big\rangle_{2,2\beta_2}^S$
by his imprimitive genus 0 Yau-Zaslow calculation in \cite{Gath}.}

For (ii), there are 30 tangent lines of $C_6$ through $p$ and another 30 through $q$.
Since we have a choice of node over the intersection of the two lines, the
contribution (ii) is
$$2\cdot30^2 = 1800\ .$$
Since the number of bitangent to $C_6$ is 324, the contribution (iii)
is 
$$2\cdot 324= 648$$
remembering again the factor 2 for the choice of node.
Hence, we calculate
$$
\big\langle \tau_{0}(\mathsf{p}),\tau_0(\mathsf{p}) \big\rangle_{2,2\beta_2}^S \ = \ 6312+1800+648 =\ 8760\,
$$
in perfect (and nontrivial) agreement with the prediction \eqref{vqq2}.

\subsection{Gromov-Witten/Pairs correspondence}
Let $S$ be a nonsingular projective $K3$ surface, and let
$$X=S\times E\, .$$
A {\em{stable pair}} $(F,s)$ is a coherent sheaf $F$ with dimension 1 
support in $X$ and a  section $s\in H^0(X,F)$ satisfying the following stability condition:
\begin{itemize}
\item $F$ is \emph{pure}, and
\item the section $s$ has zero dimensional cokernel.
\end{itemize}
To a stable pair, we associate the Euler characteristic and
the class of the support $C$ of $F$,
$$\chi(F)=n\in \mathbb{Z} \  \ \ \text{and} \ \ \ [C]=(\beta,d)\in H_2(X,\mathbb{Z})\,.$$
For fixed $n$ and $(\beta,d)$,
there is a projective moduli space of stable pairs $P_n(X,(\beta,d))$, see \cite[Lemma 1.3]{PT1}. 

The moduli space $P_n(X,(\beta,d))$ has a perfect obstruction theory of
virtual dimension 0 which yields a vanishing virtual fundamental
class. If $\beta\in \text{Pic}(S)$ is a positive
class, then the obstruction theory
can be reduced to obtain
virtual dimension 1.
 Let $\beta^\vee\in H^2(S,\mathbb{Q})$
be {\em any} class satisfying
\begin{equation}\label{fwwf2}
\langle \beta, \beta^\vee \rangle =1
\end{equation}
with respect to the intersection pairing on $S$.
For $n\in \mathbb{Z}$, we define
\begin{equation}
\label{fccf2}
\mathsf{P}^X_{n,\beta,d} = 
\int_{[P_{n}(X,(\beta,d))]^{red}} \tau_0\Big(\pi_1^*(\beta^\vee)
\cup \pi_2^*([0])\Big)\ .
\end{equation}
We follow here the notation of Section 0 for 
the projections $\pi_1$ and $\pi_2$. 
The insertions in stable pairs theory are defined in \cite{PT1}.
Definition \eqref{fccf2} is parallel to \eqref{fccf}.
As in the Gromov-Witten case, definition \eqref{fccf2} is
independent of $\beta^\vee$ satisfying \eqref{fwwf2}
by degeneration and the study of the stable pairs theory of
the rubber geometry $Y$.

Define the generating series of stable pairs invariants
for $X$ is class $(\beta,d)$ by
$$ \mathsf{P}^{X}_{\beta,d}(y)
=  \sum_{n\in \mathbb{Z}} \mathsf{P}^{X}_{n,\beta,d}\, y^{n}.$$
Elementary arguments show the moduli spaces
$P_n(X,(\beta,d))$ are empty for sufficiently negative $n$, so
${\mathsf{P}}^{X}_{\beta,d}$ is a Laurent series in $y$.
Let $$ \mathsf{N}^{X\bullet}_{\beta,d}(u)
=  \sum_{g \in \mathbb{Z}} \mathsf{N}^{X\bullet}_{g,\beta,d}\, u^{2g-2}$$
be the corresponding Gromov-Witten series for {\em disconnected}
invariants.

\vspace{10pt}
\noindent{\bf Conjecture D.}
{\em For a positive class $\beta\in \text{\em{Pic}}(S)$ and all ${d}$, 
the series  $\mathsf{P}^{X}_{\beta,d}(y)$
is the Laurent expansion of a rational function in $y$ and
$$
{\mathsf{N}}^{X\bullet}_{\beta,d}(u)\ =\ {\mathsf{P}}^{X}_{\beta,d}(y)
$$
after the variable change $y=-\exp(iu)$.}
\vspace{10pt}

The $d=0$ case of Conjecture D is exactly the Gromov-Witten/Pairs
correspondence established in \cite{PT2} for all $\beta$ as a step in the
proof of the KKV conjecture. The following result is further
evidence for Conjecture D.

\begin{prop}
 For {\em primitive} $\beta_h\in \text{\em{Pic}}(S)$ and all $d$, 
the series  $\mathsf{P}^{X}_{\beta_h,d}(y)$
is the Laurent expansion of a rational function in $y$ and
$$
{\mathsf{N}}^{X\bullet}_{\beta,d}(u)\ =\ {\mathsf{P}}^{X}_{\beta,d}(y)
$$
after the variable change $y=-\exp(iu)$.
\end{prop}

\begin{proof} We may assume $S$ is elliptically fibered as
in Section \ref{hvvc}. The reduced virtual class of
the moduli spaces of stable maps and stable pairs
under the degeneration 
\begin{equation}\label{vqqa}
S\times \mathbb{C}\ \ \leadsto\ \ R \times \mathbb{C} \ \cup\ R \times \mathbb{C}
\end{equation}
was studied in \cite{MPT}. Here, $R$ is a rational elliptic
surface. The two components of
the degeneration \eqref{vqqa} meet along
along $F\times \mathbb{C}$ where
$F\subset R$ is a nonsingular fiber of
$$\pi: R \rightarrow \p^1\ .$$
The crucial observation is that the {\em reduced} virtual class of
the moduli spaces associated to $S\times \mathbb{C}$
may be expressed in terms of the standard virtual classes
of the relative geometries \eqref{vqqa} of the degeneration.
The above argument is valid also for the degeneration
\begin{equation}\label{vqqa2}
X=S\times E \ \ \leadsto \ \ R \times E \ \cup\ R \times E\ .
\end{equation}
Since the GW/Pairs correspondence for the relative geometry
$$R\times E \, / \, F \times E$$
follows from the results of \cite{PaPix1,PaPix2}, 
we obtain the reduced correspondence for $S\times E$.
\end{proof}

\section{The full matrix}

\subsection{The Fock space}
The Fock space of the K3 surface $S$,
\begin{equation} \CF(S) = \bigoplus_{d \geq 0} \CF_d(S) = \bigoplus_{d \geq 0} H^{\ast}(S^{[d]},\BQ) \label{P1}, \end{equation}
is naturally bigraded with the $(d,k)$-th summand given by
\[ \CF_{d}^k(S) = H^{2 (k + d)}(S^{[d]},\BQ) \]
For a bihomogeneous element $\mu \in \CF_{d}^k(S)$, we let
\[ | \mu | = d, \quad \quad k(\mu) = k. \]
The Fock space
$\CF(S)$ carries a natural scalar product $\blangle \cdot\, |\, \cdot \brangle$
defined by declaring the direct sum \eqref{P1} orthogonal and setting 
\[ \blangle \mu\, |\, \nu \brangle = \int_{S^{[d]}} \mu \cup \nu \]
for every $\mu, \nu \in H^{\ast}(S^{[d]},\BQ)$.
For $\alpha, \alpha' \in H^{\ast}(S,\BQ)$, we also write
\[ \langle \alpha, \alpha' \rangle = \int_S \alpha \cup \alpha'. \]
If $\mu, \nu$ are bihomogeneous, then $\langle \mu | \nu \rangle$ is nonvanishing only in the case $|\mu| = |\nu|$ and $k(\mu) + k(\nu) = 0$.

For all $\alpha \in H^{\ast}(S,\BQ)$ and $m \neq 0$, the Nakajima operators $\Fp_m(\alpha)$ act on $\CF(S)$
bihomogeneously of bidegree $(-m, k(\alpha))$,
\[ \Fp_{m}(\alpha) : \CF_d^k \ra \CF_{d-m}^{k+ k(\alpha)} \,. \]
The commutation relations
\begin{equation} [ \Fp_{m}(\alpha), \Fp_{m'}(\alpha') ] = -m \delta_{m + m',0} 
\blangle \alpha, \alpha'\brangle\, {\rm id}_{\CF(S)}, \label{N1} \end{equation}
are satisfied 
for all $\alpha, \alpha' \in H^{\ast}(S)$ and all $m, m' \in \BZ \setminus 0$.

The inclusion of the diagonal $X \subset X^m$ induces a map
\[ \tau_{\ast m} : H^{\ast}(X,\BQ) \ra H^{\ast}(X^m,\BQ) \stackrel{\sim}{=} H^{\ast}(X,\BQ)^{\otimes m} \, .\]
For $\tau_{\ast} = \tau_{\ast 2}$, we have 
\[ \tau_{\ast}(\alpha) = \sum_{i,j} g^{ij} \, (\alpha \cup \gamma_i) \otimes \gamma_j, \]
where $\{ \gamma_i \}$ is a basis of $H^{\ast}(X)$ and $g^{ij}$ is the inverse of the intersection matrix $g_{ij} = \blangle \gamma_i, \gamma_j \brangle$.

For $\gamma \in H^{\ast}(S,\BQ)$, define the degree zero Virasoro operator
\[ L_0(\gamma) = - \frac{1}{2} \sum_{k \in \BZ \setminus 0} : \Fp_k \Fp_{-k} : \tau_{\ast}(\gamma) = - \sum_{k \geq 1} \sum_{i,j} g^{ij} \Fp_{-k}(\gamma_i \cup \gamma) \Fp_k(\gamma_j) \, , \]
where $: -- :$ is the normal ordered product, see \cite{LH}.
For $\alpha \in H^{\ast}(S,\BQ)$, we have then
\[ [ \Fp_{k}(\alpha), L_0(\gamma) ] = k \Fp_{k}(\alpha \cup \gamma). \]
Let $\e \in H^{\ast}(S)$ denote the unit. The restriction of $L_0(\gamma)$ to $\CF_d(S)$,
\[ L_0(\gamma)|_{\CF_d(S)} : H^{\ast}(S^{[d]},\BQ) \ra H^{\ast}(S^{[d]},\BQ) \]
is the cup product by the class
\[ D_d(\gamma) = \frac{1}{(d-1)!} \Fp_{-1}(\gamma) \Fp_{-1}(\e)^{d-1} \in H^{\ast}(S^{[d]},\BQ) \]
of subschemes incident to $\gamma$, see \cite{LT}.
In the special case, $\gamma = \e$, $L_0 = L_0(\e)$ is the energy operator,
\[ L_0(\e)|_{\CF_d(S)} = d \cdot {\rm id}_{\CF_d(S)} \,. \]

\noindent Finally, define Lehn's diagonal operator \cite{LT}:
\[ \partial = - \frac{1}{2} \sum_{i,j \geq 1} ( \Fp_{-i} \Fp_{-j} \Fp_{i+j} + \Fp_i \Fp_j \Fp_{-(i+j)} ) \tau_{3 \ast}( [X] ) \,. \]
For $d \geq 2$, $\partial$ acts on $\CF_d(S)$ by the cup product with $-\frac{1}{2} \Delta_{S^{[d]}}$, where
\[ \Delta_{S^{[d]}} = \frac{1}{(n-2)!} \Fp_{-2}(\e) \Fp_{-1}(\e)^{n-2} \]
denotes the class of the diagonal in $S^{[d]}$.


\subsection{Quantum multiplication}
Let $S$ be an elliptic $K3$ surface with section class $B$ and fiber class $F$. For $h\geq 0$, let
$$\beta_h = B + hF\, .$$ 
We will define quantum multiplication on $\CF(S)$ with respect to the classes $\beta_h$.

For $\alpha_1, \dots, \alpha_m \in H^{\ast}(S^{[d]},\BQ)$, define the quantum bracket
\begin{multline}
\blangle \alpha_1, \dots, \alpha_m \brangle_q^{S^{[d]}} = \\
\sum_{h \geq 0} \sum_{k \in \BZ} y^k q^{h-1} \int_{[\M_{0,m}(S^{[d]}, C(\beta_h) + kA)]^{\text{red}}} \ev_1^{\ast}(\alpha_1) \cdots 
 \ev_m^{\ast}(\alpha_m) \label{P2}
\end{multline}
\noindent as an element of 
$\BQ((y))((q))$.{\footnote{By standard arguments \cite{GO}, the
moduli space $\M_{0,m}(S^{[d]}, C(\beta_h) + kA)$ is empty for $k$
sufficiently negative.}}
Because $d$ is determined by the $\alpha_i$, we often omit $S^{[d]}$. 
The multilinear
pairing $\langle \cdots \rangle$ extends naturally to the Fock space by declaring the pairing 
orthogonal with respect to \eqref{P1}.

Let $\epsilon$ be a formal parameter with $\epsilon^2 = 0$. For 
$$a,b,c \in H^{\ast}(S^{[d]},\BQ)\, ,$$ define the (primitive) quantum product $\ast$ by
\[ \langle a\ |\ b \ast c \brangle\ =\ \blangle a\ |\ b \cup c \brangle + \epsilon \cdot \blangle a, b, c \brangle_q \,. \]
As $\blangle \cdots \brangle_q$ takes values in 
$\BQ((y))((q))$, 
the product $\ast$ is defined over the ring
$$H^{\ast}(S^{[d]},\BQ) \otimes \BQ((y))((q)) \otimes \BQ[\epsilon]/\epsilon^2\, .$$
By the WDVV equation in the reduced case (see \cite[Appendix 1]{GO}), $\ast$ is associative. 
We extend $\ast$ to an associative product on $\CF(S)$ by $b \ast c = 0$ 
whenever $b$ and $c$ are in different summands of \eqref{P1}.

The parameter $\epsilon$ has to be introduced since we use \emph{reduced}
Gromov-Witten theory to define the bracket \eqref{P2}. It can be thought of as an infinitesimal virtual weight on the canonical class $K_{S^{[n]}}$ 
 and corresponds in the toric case (see \cite{MOH,OPH}) to the equivariant parameter $(t_1 + t_2)$ mod $(t_1 + t_2)^2$.

We are mainly interested in the 2-point quantum operator 
\[ \CE^{\Hilb} : \CF(S) \otimes \BQ((y))((q)) \ra \CF(S) \otimes \BQ((y))((q)) \]
defined by the bracket
\[ \blangle a\ |\ \CE^{\Hilb} b \brangle = \blangle a , b \brangle_q \]
and extended $q$ and $y$ linearly.
Because $\M_{0,2}(S^{[d]},\alpha)$ has reduced virtual dimension $2d$, $\CE^{\Hilb}$ is a self-adjoint operator of bidegree $(0,0)$.

Let $D_1, D_2 \in H^2(S^{[d]},\BQ)$ be divisor classes. 
By associativity and commutativity,
\begin{equation} D_1 \ast ( D_2 \ast a ) = D_2 \ast ( D_1 \ast a ) \label{12333} \end{equation}
for all $a$.
By the divisor axiom, we have
\begin{eqnarray*}
D_d(\gamma) \ast \cdot\, \Big|_{\CF_d(S)}              & = & \Big( L_0(\gamma) + \epsilon\, \Fp_{0}(\gamma) \CE^{\Hilb} \Big)\Big|_{\CF_d(S)} \\
\frac{-1}{2} \Delta_{S^{[d]}} \ast \cdot\, \Big|_{\CF_d(S)} & = & \Big( \partial    + \epsilon\, y \frac{d}{d y} \CE^{\Hilb} \Big)\Big|_{\CF_d(S)}\, .
\end{eqnarray*}
for every $\gamma \in H^2(S,\BQ)$. 
Here, $\frac{d}{dy}$ is formal differentiation with respect to the variable $y$,
and $\Fp_0(\gamma)$ for $\gamma \in H^{\ast}(S)$ is the degree $0$ Nakajima operator defined by the following condition\footnote{This definition precisely matches the action of the extended Heisenberg algebra $\blangle \Fp_k(\gamma) \brangle, k \in \BZ$ on the full Fock space
$\CF(S) \otimes \BQ[ H^{\ast}(S,\BQ) ]$
under the embedding $q^{h-1} \mapsto q^{B + hF}$, see \cite[section 6.1]{KY}.}:
\[ [ \Fp_0(\gamma), \Fp_{m}(\gamma') ] = 0 \]
for all $\gamma' \in H^{\ast}(S), m \in \BZ$ and
\[ \Fp_0(\gamma)\, q^{h-1} y^k \, 1_S = \blangle \gamma, \beta_h \brangle q^{h-1} y^k\, 1_S \,. \]

After specializing $D_i$, we obtain the main commutator relations for $\CE^{\Hilb}$ on $\CF(S)$,
\begin{equation} \label{P4}
\begin{aligned}
\Fp_0(\gamma)\, [ \CE^{\Hilb}, L_0(\gamma') ]  & = \Fp_0(\gamma')\, [ \CE^{\Hilb}, L_0(\gamma) ] \\
\Fp_0(\gamma)\, [ \CE^{\Hilb}, \partial ]  & = y \frac{d}{dy}\, [ \CE^{\Hilb}, L_0(\gamma) ]  \,,
\end{aligned}
\end{equation}
for all $\gamma, \gamma' \in H^2(S,\BQ)$. The equalities \eqref{P4} are true
only after restricting to $\CF(S)$, and not on all of $\CF(S) \otimes \BQ((y))((q))$ by
definition of $\Fp_0(\gamma)$ and $y \frac{d}{dy}$.

Equation \eqref{P4} shows the commutator of $\CE^{\Hilb}$ with a divisor intersection operator 
to be essentially independent of the divisor.

\subsection{The operators $\CE^{(r)}$}
Let
\begin{equation} \varphi_{m,\ell}(y,q)\, \in \BC((y^{1/2}))[[q]] \label{006} \end{equation}
be fixed power series that satisfy the symmetries
\begin{equation}
\begin{aligned} \label{symmetries_phi}
 \varphi_{m,\ell} & = - \varphi_{-m, -\ell} \\
 \ell \varphi_{m,\ell} & = m \varphi_{\ell, m} \,.
\end{aligned}
\end{equation}
for all $(m,\ell) \in \BZ^2 \setminus 0$.
Depending on the functions \eqref{006}, define for all $r \in \BZ$ operators
\[ \CE^{(r)} : \CF(S) \otimes \BC((y^{1/2}))((q)) \ra \CF(S) \otimes \BC((y^{1/2}))((q)) \]
by the following recursion relations:

\vspace{8pt}
\noindent{\bf Step 1.}
For all $r \geq 0$,
\begin{equation*}
\CE^{(r)} \Big|_{\CF_0(S) \otimes \BC((y^{1/2}))((q))} = \frac{\delta_{0r}}{F(y,q)^2 \Delta(q)} \cdot \id_{\CF_0(S) \otimes \BC((y^{1/2}))((q))},
\end{equation*}
where $F(y,q)$ and $\Delta(q)$ are the functions defined in section $2$ considered as formal expansions in the variables $y$ and $q$. 

\vspace{8pt}
\noindent{\bf Step 2.}
For all $m \neq 0, r \in \BZ$,
{\small{\[ [ \Fp_{m}(\gamma), \CE^{(r)} ] = \sum_{\ell \in \BZ} \frac{\ell^{k(\gamma)}}{m^{k(\gamma)}} : \Fp_\ell(\gamma) \CE^{(r+m-\ell)} : \, \varphi_{m,l}(y,q) \]}}
Here $k(\gamma)$ denotes the shifted complex cohomological degree of $\gamma$,
\[ \gamma \in H^{2(k(\gamma) + 1)}(S;\BQ) \,, \]
and $: -- :$ is a variant of the normal ordered product defined by
\[ : \Fp_\ell(\gamma) \CE^{(k)} : = \begin{cases}
                             \Fp_{\ell}(\gamma) \CE^{(k)} & \text{ if } \ell \leq 0 \\
                             \CE^{(k)} \Fp_{\ell}(\gamma) & \text{ if } \ell > 0 \,.
                            \end{cases}
\]

The two steps uniquely determine the operators $\CE^{(r)}$. 
It follows from the symmetries \eqref{symmetries_phi},
that $\CE^{(r)}$ respects the Nakajima commutator relations \eqref{N1}. 
 Hence 
$\CE^{(r)}$ acts on $\CF(S)$ and is therefore well defined.
By definition, it is an operator of bidegree $(-r,0)$, which is $y$-linear, but \emph{not} $q$ linear.
%

\vspace{8pt}
\noindent{\bf Conjecture E.}
{\em There exist unique functions $\varphi_{m,\ell}$ for $(m,\ell) \in \BZ^2 \setminus 0$ that satisfy:
\begin{enumerate}
\item[(i)] Initial conditions:
\[ \varphi_{1,1} = G(y,q) - 1, \quad \varphi_{1,0} = -i F(y,q), \quad \varphi_{1,-1} = -\frac{1}{2} q \frac{d}{dq}( F(y,q)^2 ) \,. \]
\vspace{5pt}
\item[(ii)]
$\CE^{(0)}$ satisfies the WDVV equations:
\begin{align*}
\Fp_0(\gamma)\, [ \CE^{(0)}, L_0(\gamma') ] & = \Fp_0(\gamma')\, [ \CE^{(0)}, L_0(\gamma) ] \\
\Fp_0(\gamma)\, [ \CE^{(0)}, \partial ]     & = y \frac{d}{dy}\, [ \CE^{(0)}, L_0(\gamma) ]
\end{align*}
on $\CF(S)$ for all $\gamma, \gamma' \in H^2(S,\BQ)$. 
\end{enumerate}}

\vspace{8pt}
Conjecture E has been checked numerically on $\CF_d(S)$ for $d \leq 5$.
The functions $\varphi_{m,\ell} + {\rm sgn}(m) \delta_{ml}$ are expected to be quasi Jacobi forms with weights and index for all non-vanishing cases given by the following table:
\begin{center}
\begin{longtable}{|l | c | c | }
\hline
 & \text{index} & \text{weight} \\[0.5ex]
\hline 
$m \neq 0, \ell \neq 0$ & $\frac{|m| + |\ell|}{2}$ & $0$ \\[0.5ex]
$m \neq 0, \ell = 0$ & $\frac{|m|}{2}$ & $-1$\\
\hline
\end{longtable}
\end{center}
\vspace{-0.5cm}
The first values of $\varphi_{m,l}$ are:
\begin{align*}
\varphi_{2,2} + 1& = 2 K^{4} \cdot \left( J_1^{2} \wp(z) - \frac{1}{12} J_1^{2} E_2 + \frac{3}{2} \wp(z)^{2} + J_1 \partial_{z}(\wp(z)) - \frac{1}{96} E_4 \right) \\
\varphi_{2,1} & = 2 K^{3} \cdot \left( J_1 \wp(z) - \frac{1}{12} J_1 E_2 + \frac{1}{2} \partial_z(\wp(z)) \right) \\
\varphi_{2,0} & = -2 \cdot J_1 \cdot K^{2} \\
\varphi_{2,-1} & = -\frac{4}{3} \cdot K^{3} \cdot \left( J_1^{3} - \frac{3}{2} J_1 \wp(z) - \frac{1}{8} J_1 E_2 - \frac{1}{4} \partial_z(\wp(z)) \right) \\
\varphi_{2,-2} & = 2 J_1 \cdot K^{4} \cdot \left( J_1^{3} - 2 J_1 \wp(z) - \frac{1}{12} J_1 E_2 - \frac{1}{2} \partial_z(\wp(z)) \right),
\end{align*}
where $K = iF$ and $J_1 = \partial_z(\log(F))$.

Conjectures E and F (below) were first
proposed in different but equivalent forms in \cite{GO}.

\subsection{Further conjectures}

Let $L_0$ be the energy operator on $\CF(S)$. We define the operator
\[ G^{L_0} : \CF(S) \otimes \BQ((y))((q)) \ra \CF(S) \otimes \BQ((y))((q)) \]
by
\[ G^{L_0}( \mu ) = G(y,q)^{|\mu|} \cdot \mu \]
for any homogeneous $\mu \in \CF(S)$.

\vspace{8pt}
\noindent{\bf Conjecture F.} For $S$ an elliptic K3 surface we have on $\CF(S)$
\[ \CE^{\Hilb} \ = \ \CE^{(0)} - \frac{1}{F^2 \Delta} G^{L_0} \, . \]
\vspace{4pt}

We stated the Conjecture for the elliptic $K3$ surface $S$ with respect to the classes 
$$\beta_h = B + hF\, .$$
 By extracting the $q^{h-1}$-coefficient and deforming the $K3$ surface, we obtain the $2$-point invariants for any pair $(S',\beta')$ of a $K3$ surface $S'$ and a primitive curve class $\beta'$ of square $2h-2$.

The trace 
 of the operator $\frac{1}{F^2 \Delta} G^{L_0}$ on the Fock space $\CF(S)$ is
\[ \Tr_{\CF(S)} \frac{1}{F^2 \Delta} \tilde{q}^{L_0 - 1} G^{L_0} = \frac{1}{F^2 \Delta} \sum_{d \geq 0} G^{d} \chi(S^{[d]}) \tilde{q}^{d-1}. \]
By G\"ottsche's formula \cite{Gottsche}, we obtain precisely the correction term \eqref{cxxc}.
Hence Conjectures A and F together imply
\begin{equation*} \Tr_{\CF(S)} \tilde{q}^{L_0 - 1} \CE^{(0)} = -\frac{1}{\chi_{10}(\Omega)} \,. \end{equation*}
The above equation is a purely algebraic statement about the operator $\CE^{(0)}$.

Let $P_n(Y, (\beta_h,d))$ be 
the moduli space of stable pairs on the straight rubber
geometry
$$Y=S \times R$$
defined in Section 1.
The reduced virtual dimension of the moduli space $P_n(Y, (\beta_h,d))$
is $2d$.
Let
\[ \ev_i : P_n(Y, (\beta_h,d)) \rightarrow
  S^{[d]}, \quad i = 0,\infty \]
be the boundary maps\footnote{For $d=0$, we take
$S^{[0]}$ to be a point.}.

Define the bidegree $(0,0)$ operator
\[ \CE^{\text{Pairs}} : \CF(S) \otimes \BQ((y))((q)) \ra \CF(S) \otimes \BQ((y))((q)) \]
on $\CF_d(S)$ by
\[ \big\langle \mu\ \big|\ \CE^{\text{Pairs}} \nu \big\rangle = 
\sum_{h \geq 0} \sum_{n \in \BZ} y^n q^{h-1} \int_{[P_n(Y, (\beta_h,d))]^{\text{red}}} 
\ev_0^{\ast}(\mu) \cup \ev_\infty^{\ast}(\nu)\, . \]

\vspace{8pt}
\noindent{\bf Conjecture G.} On the elliptic $K3$ surface $S$,
\[ \CE^{\Hilb} + \frac{1}{F^2\Delta} G^{L_0} = y^{-L_0} \CE^{\text{Pairs}}. \]
\vspace{8pt}

We have stated Conjecture G as relating $\CE^{\Hilb}$ and $\CE^{\text{Pairs}}$.
Combining Conjectures F and G leads to the direct prediction on $\CF(S)$:
$$\CE^{\text{Pairs}}\ =\ y^{L_0}\, \CE^{(0)}\ .$$

A conjecture relating the stable pairs theory of $S \times R$ to the Gromov-Witten side
is formulated exactly as in  Conjecture D.
We can express the conjectural relationship between the different theories by the triangle:
\begin{center}
\begin{tikzpicture}
 \draw[very thick]
   (0,0) node[below left, align = center]{Gromov-Witten\\theory\\of $K3 \times \BP^1$}
-- (1,1.732) node[above, align = center]{Quantum cohomology\\of $\Hilb^d(K3)$}
-- (2,0) node[below right, align = center]{Stable pairs\\theory\\of $K3 \times \BP^1$}
-- (0,0);
\end{tikzpicture}
\end{center}
\subsection{Three examples} 


\noindent{\bf(i)} Let $F$ be the fiber of the elliptic fibration. Then, we have
\begin{align*}
\blangle \Fp_{-1}(F)^d 1_S\ |\ \CE^{(0)} \Fp_{-1}(F)^d 1_S \brangle & = (-1)^d \blangle 1_S\ |\ \Fp_{1}(F)^d \CE^{(0)} \Fp_{-1}(F)^d 1_S \brangle \\
& = (-1)^d \blangle 1_S\ |\ \Fp_0(F)^d \CE^{(d)} \varphi_{1,0}^d \Fp_{-1}(F)^d 1_S \brangle \\
& = (-1)^d \blangle 1_S\ |\ \Fp_0(F)^{2d} \CE^{(0)} (-1)^d \varphi_{1,0}^{d} \varphi_{-1,0}^d 1_S \brangle \\
& = \frac{\varphi_{1,0}^{d} \varphi_{-1,0}^d}{F(y,q)^2 \Delta(q)} \\
& = \frac{F(y,q)^{2d-2}}{\Delta(q)}
\end{align*}
in agreement with Theorem \ref{MThm0}. We have used $\Fp_0(F) = 1$ above.

\vspace{8pt}
\noindent{\bf(ii)} Let $W = B + F$. Then $W^2 = 0$ and $\blangle W, \beta_h \brangle = h-1$. In particular
$\Fp_0(W)$ acts as $\partial_{\tau} = q \frac{d}{dq}$.
We have
\begin{align*}
\blangle \Fp_{-1}(W)^d 1_S\ |\ \CE^{(0)} \Fp_{-1}(W)^d 1_S \brangle
& = (-1)^d \blangle 1_S\ |\ \Fp_0(W)^d \CE^{(d)} \varphi_{1,0}^d \Fp_{-1}(W)^d 1_S \brangle \\
& = \blangle 1_S\ |\ \Fp_0(W)^{2d} \CE^{(0)} \varphi_{1,0}^{d} \varphi_{-1,0}^d 1_S \brangle \\
& = \partial_{\tau}^{2d} \left( \frac{\varphi_{1,0}^{d} \varphi_{-1,0}^d}{F(y,q)^2 \Delta(q)} \right)\\
& = \partial_{\tau}^{2d} \left( \frac{F(y,q)^{2d-2}}{\Delta(q)} \right).
\end{align*}

\vspace{8pt}
\noindent{\bf(iii)}
Let $\pt \in H^4(S;\BZ)$ be the class of a point. For all $d \geq 1$, let
\[ C(F) = \Fp_{-1}(F) \Fp_{-1}(\pt)^{d-1} 1_S \in H_2(S^{[d]},\BZ)\, . \]
Then, assuming Conjecture F,
{\allowdisplaybreaks
\begin{align*}
\blangle C(F) \brangle_q & = \blangle C(F), D_d(F) \brangle_q \\
& = \frac{1}{(d-1)!} \blangle \Fp_{-1}(F) \Fp_{-1}(\pt)^{d-1} 1_S \ |\ \CE^{(0)} \Fp_{-1}(F) \Fp_{-1}(e)^{d-1} 1_S \brangle \\
& = \frac{1}{(d-1)!}  \blangle \Fp_{-1}(\pt)^{d-1} 1_S \ |\ \CE^{(0)} \varphi_{1,0} \varphi_{-1,0} \Fp_{-1}(e)^{d-1} 1_S \brangle \\
& = \frac{(-1)^{d-1}}{(d-1)!}  \blangle 1_S \ |\ \CE^{(0)} \varphi_{1,0} \varphi_{-1,0} (\varphi_{1,1} + 1)^{d-1} \Fp_{1}(\pt)^{d-1} \Fp_{-1}(e)^{d-1} 1_S \brangle \\
& = \frac{ \varphi_{1,0} \varphi_{-1,0} (\varphi_{1,1} + 1)^{d-1} }{F(y,q)^2 \Delta(q)} \\
& = \frac{ G(y,q)^{d-1} }{\Delta(q)}
\end{align*}
}
in full agreement with the first part of Theorem 2 in \cite{GO}.

\subsection{The $\CA_1$ resolution.}
Let
\[ \CE_B^{\text{Pairs}} = [q^{-1}] \CE^{\text{Pairs}} \quad \text{ and } \quad \CE_B^{\Hilb} = [q^{-1}] \CE^{\Hilb} \]
be the restriction of $\CE^{\Hilb}$ and $\CE^{\text{Pairs}}$ to the case of the class $\beta_0 = B$.\footnote{We denote with $[q^{-1}]$ the operator that extracts the $q^{-1}$ coefficient.}
The corresponding local case was considered before in \cite{MODT, MOH}.
Define operators $\CE_{B}^{(r)}$ by
\begin{align*}
\blangle 1_S\ |\ \CE_{B}^{(r)} 1_S \brangle & = \frac{y}{(1+y)^2} \delta_{0r} \\
[ \Fp_{m}(\gamma), \CE_B^{(r)} ] & = \blangle \gamma, B \brangle  \left( (-y)^{-m/2} - (-y)^{m/2} \right) \CE_B^{(r+m)}
\end{align*}
for all $m \neq 0$ and $\gamma \in H^{\ast}(S)$, see \cite[Section 5.1]{MOH}.
Translating the results of \cite{MODT,MOH} to the 
$K3$ surface leads to the following evaluation.

\begin{thm} \label{8765} We have
\[ \CE_B^{\Hilb} + \frac{y}{(1+y)^2} {\rm Id_{\CF(S)}} = y^{-L_0} \CE_B^{\text{Pairs}} = \CE_{B}^{(0)} \,. \]
\end{thm}

\noindent From numerical experiments \cite{GO}, we expect the expansions
\begin{alignat*}{2}
 \varphi_{m,0} & = \left( (-y)^{-m/2} - (-y)^{m/2} \right) + O(q) \quad \quad && \text{ for all } m \neq 0 \\
 \varphi_{m,\ell} & = O(q) && \text{ for all } \ell \neq 0, m \in \BZ \,.
\end{alignat*}
Because of
\[ [q^{-1}] \frac{G^{L_0}}{F^2 \Delta} = \frac{y}{(1+y)^2} \text{Id}_{\CF(S)}\, , \]
we find conjectures F and G to be in complete agreement with Theorem \ref{8765}.

\noindent From Theorem \ref{8765}, we obtain the interesting relation
\[ \Tr_{\CF(S)} q^{L_0 - 1} \CE_B^{\text{Pairs}} = \frac{1}{y + 2 + y^{-1}}\, \frac{1}{q} \prod_{m \geq 1} 
\frac{1}{(1 + y^{-1} q^m)^2 (1 - q^m)^{20} (1 + y q^m)^2}\ . \]
By the symmetry of $\chi_{10}$ in the variables $q$ and $\tilde{q}$, we obtain agreement with Conjecture A.

%

\section{Motivic theory}
Let $S$ be a nonsingular projective $K3$ surface, and let 
 $\beta \in \text{Pic}(S)$ be a positive
class (with respect to any ample polarization).
We will assume $\beta$ is {\em irreducible} (not expressible
as a sum of effective classes).

To unify our study with \cite{KKP}, we end the paper with a discussion of
the motivic stable pairs invariants of 
$$X=S \times E\ $$
in class $(\beta,d)$.
Following the conjectural perspective of \cite{KKP},
we assume the Betti realization of the motivic invariants of
$X$ is {\em both} well-defined {\em and} independent of
deformations of $S$ for which $\beta$  remains algebraic
and irreducible. 

We define a generating function $\mathsf{Z}$ of the 
Betti realizations of the motivic stable pairs theory of $X$
in classes $(\beta_h,d)$ where $\beta_h$ is irreducible
and satisfies 
$$\langle \beta_h, \beta_h \rangle = 2h-2\, . $$
The series $\mathsf{Z}$ depends upon the variables
$y$, $q$, $\tilde{q}$ just as before and a new variable $u$
for the virtual Poincare polynomial:
\[ \mathsf{Z}(u,y,q,\tilde{q}) = \frac{1}{u^{-1}+2+u}
\sum_{h\geq 0} \sum_{d\geq 0} \mathsf{H}\big(P_n(S\times E, (\beta_h,d))\big)
\, y^n q^{h-1} \tilde{q}^{d-1} \ .\]
Here we follow the notation of \cite[Section 6]{KKP} for the
normalized virtual Poincar\'e polynomial 
\[ \mathsf{H}\big(P_n(S\times E, (\beta_h,d))\big)\in \mathbb{Z}[u,u^{-1}] \,. \]
In the definition of $\mathsf{Z}$,
the prefactor $\frac{1}{u^{-1}+2+u}$ (the reciprocal of the 
normalized Poincar\'e polynomial of $E$) quotients
by the translation action of $E$ on $P_n(S\times E, (\beta_h,d))$.

Because of the $u$ normalization, we have the following symmetry of $\mathsf{Z}$
in the variable $u$:
\begin{enumerate}
\item[(i)] $\mathsf{Z}(u,y,q,\tilde{q})= \mathsf{Z}(u^{-1},y,q,\tilde{q})$\ .
\end{enumerate}
Two further properties which constrain $\mathsf{Z}$
are: 
\begin{enumerate}
\item[(ii)] {\em the specialization $u=-1$ must recover the
stable pairs invariants (determined by Conjectures A and D),
$$\mathsf{Z}(-1,y,q,\tilde{q})= - \frac{1}{\chi_{10}}\ ,$$}
\item[(iii)] {\em the coefficient of ${\tilde{q}}^{-1}$
must specialize to the motivic series of \cite[Section 4]{KKP},
\begin{multline*} 
 {\left(uy-1\right)\left(u^{-1}-{y}^{-1}\right)} \cdot 
\text{\em Coeff}_{\, {\tilde{q}^{-1}}}\Big(\mathsf{Z}(u,y,q,\tilde{q})\Big)
= \\
\prod_{n=1}^\infty
 \frac{1}{
(1-u^{-1}y^{-1} q^n)(1-u^{-1}y q^n) (1-q^n)^{20}(1-uy^{-1}q^n)
(1-u yq^n)} \, ,
\end{multline*}
obtained from the Kawai-Yoshioka calculation \cite{KY}.}
\end{enumerate}



To obtain further constraints, we study the virtual 
Poincar\'e polynomial
$$\mathsf{H}\Big(P_{1-h-d}(S\times E, (h,d))/E\Big) \in \mathbb{Z}[u,u^{-1}]\ $$
which arises as
\begin{equation} \label{jjjj}
\text{Coeff}_{\, y^{1-h-d}q^{h-1}\tilde{q}^{d-1}}\big(\mathsf{Z}\big)\ .
\end{equation}
For $q^{h-1}\tilde{q}^{d-1}$, the coefficient \eqref{jjjj} corresponds to the 
the {\em lowest} order term in $y$. We have an isomorphism of
the moduli spaces{\footnote{We follow the notation of \cite{KKP} for the
moduli of
stable pairs $P_{n}(S,h)$ on $K3$ surfaces.}},
$$P_{1-h-d}(S \times E,(h,d))/E \cong P_{1-h+d}(S,h)\ .$$
Hence, we obtain a fourth constraint for $\mathsf{Z}$.

\vspace{+8pt}
\begin{enumerate}
\item[(iv)] {\em $\text{\em Coeff}_{\, y^{1-h-d}q^{h-1}\tilde{q}^{d-1}}\big(\mathsf{Z}\big)$
equals the $y^{1-h+d}q^{h-1}$ coefficient of
{\tiny{$${\frac{1}{\left(uy-1\right)\left(u^{-1}-{y}^{-1}\right)}
\prod_{n=1}^\infty
 \frac{1}{
(1-u^{-1}y^{-1} q^n)(1-u^{-1}y q^n) (1-q^n)^{20}(1-uy^{-1}q^n)
(1-u yq^n)}}$$}}}
\end{enumerate}
\vspace{+10pt}

The function $-\frac{1}{\chi_{10}}$
has a basic symmetry in the variables $q$ and $\tilde{q}$.
As stated, condition (iv) is not symmetric in $q$ and $\tilde{q}$.
However, the symmetry 
\begin{equation*} 
\text{Coeff}_{\, y^{1-h-d}q^{h-1}\tilde{q}^{d-1}}\big(\mathsf{Z}\big)=
\text{Coeff}_{\, y^{1-h-d}q^{d-1}\tilde{q}^{h-1}}\big(\mathsf{Z}\big)
\end{equation*}
can be easily verified from (iv).
Unfortunately, further calculations show that
the symmetry 
in the variables $q$ and $\tilde{q}$ appears
{\em not} to lift to the motivic theory.

A basic question is to specify the modular
properties of $\mathsf{Z}$.
We hope conditions (i)-(iv) together with
the modular properties of
 $\mathsf{Z}$ will uniquely determine $\mathsf{Z}$.
There is every reason to expect the function $\mathsf{Z}$
will be beautiful.



\begin{thebibliography}{99}
\bibitem{Beau} 
A.~Beauville, \textsl{Vari\'et\'es {K}\"ahleriennes dont la premi\`ere classe
  de {C}hern est nulle},
\newblock J. Differential Geom. \textbf{18} (1984), 755--782.


\bibitem{Beh} K. Behrend, {\em The product formula for Gromov-Witten
invariants}, J. alg. Geom. \textbf{8} (1999), 529--541.

\bibitem{bghz} J. Bruinier, G. van der Geer, G. Harder, and D. Zagier,
{\em The 1-2-3 of modular forms}, Springer Verlag: Berlin, 2008.

\bibitem{jb} J. Bryan, {\em The Donaldson-Thomas theory of $K3\times E$
via the topological vertex}, arXiv:1504.02920.

\bibitem{BL} J. Bryan and C. Leung, {\em The enumerative geometry of $K3$ surfaces and
modular forms}, JAMS \textbf{13} (2000), 549--568.


\bibitem{BP}
J.~Bryan and R.~Pandharipande,
\newblock {\em The local {G}romov-{W}itten theory of curves}, J.
  Amer. Math. Soc.  {\bf 21} (2008), 101--136.

\bibitem{DMZ} A.~Dabholkar, S.~Murthy, and D.~Zagier, {\em Quantum
black holes, wall crossing, and mock modular forms}, arXiv:1208.4074.

\bibitem{EZ} 
M.~Eichler and D.~Zagier,
\newblock \textsl{ The theory of {J}acobi forms}, volume~55 of \textsl{
  Progress in Mathematics},
\newblock Birkh\"auser Boston Inc., Boston, MA, 1985.

\bibitem{FP}
C.~Faber and R.~Pandharipande.
\newblock {\em Hodge integrals and {G}romov-{W}itten theory}, Invent. Math.,
  {\bf 139} (2000), 173--199.

\bibitem{growi} A.~Gathmann, {\em GROWI: a C++ program for Gromov-Witten invariants},

$\text{\tiny{www.mathematik.uni-kl.de/agag/mitglieder/professoren/gathmann/software/growi/}}$ .

\bibitem{Gath} A.~Gathmann, {\em The number of plane conics that are 5-fold tangent to
a given curve}, Compos. Math. \textbf{2} (2005), 487--501.

\bibitem{Gottsche} L. G\"ottsche, {\em The Betti numbers of the Hilbert scheme of points on a smooth projective surface}, Math. Ann. \textbf{286} (1990), 
193--207.

\bibitem{GN} 
V.~A. Gritsenko and V.~V. Nikulin, \textsl{Siegel automorphic form corrections
  of some {L}orentzian {K}ac-{M}oody {L}ie algebras},
\newblock Amer. J. Math. \textbf{ 119} (1997), 181--224.

\bibitem{KKP} S. Katz, A. Klemm, and R. Pandharipande, 
{\em On the motivic stable pairs invariants of $K3$ surfaces},
arXiv:1407.3181.


\bibitem{KKV} S. Katz, A. Klemm, and C. Vafa, {\em M-theory, topological
strings, and spinning black holes}, Adv. Theor. Math. Phys. {\bf 3} (1999),
1445--1537.

\bibitem{KK3} T. Kawai, {\em $K3$ surfaces, Igusa cusp form, and
string theory}, hep-th/9710016.

\bibitem{KY} T. Kawai and K Yoshioka, {\em String partition functions
and infinite products}, Adv. Theor. Math. Phys. {\bf 4} (2000), 397--485.


\bibitem{LH}
M.~Lehn,
\newblock Lectures on {H}ilbert schemes,
\newblock in \textsl{ Algebraic structures and moduli spaces}, volume~38 of
  \textsl{ CRM Proc. Lecture Notes}, pages 1--30, Amer. Math. Soc., Providence,
  RI, 2004.

\bibitem{LT} M. Lehn, 
{\em Chern classes of tautological sheaves on {H}ilbert schemes of
  points on surfaces},
\newblock Invent. Math. \textbf{ 136} (1999), 157--207.


\bibitem{MNOP1}
D.~Maulik, N.~Nekrasov, A.~Okounkov, and R.~Pandharipande,
\newblock {\em 
Gromov-{W}itten theory and {D}onaldson-{T}homas theory. {I},} 
{Compos. Math.}, {\bf 142} (2006), 1263--1285.


\bibitem{DM} D. Maulik, {\em Gromov-Witten theory of
$A_n$-resolutions}, Geom. and Top. {\bf 13} (2009), 1729--1773. 

\bibitem{MODT}
D.~Maulik and A.~Oblomkov, {\em Donaldson-Thomas theory of $A_n\times P^1$},
Compos. Math. 145 (2009), no. 5, 1249–1276.

\bibitem{MOH} 
D.~Maulik and A.~Oblomkov, {\em Quantum cohomology of the {H}ilbert scheme
  of points on {$A_n$}-resolutions},
\newblock JAMS \textbf{22} (2009), 1055--1091.



\bibitem{gwnl} D. Maulik and R. Pandharipande, {\em 
Gromov-Witten theory and Noether-Lefschetz theory}, in {\em A 
celebration of algebraic geometry}, Clay Mathematics Proceedings {\bf 18},
469--507,
AMS (2010).

\bibitem{MPT} D. Maulik, R. Pandharipande, and R. Thomas, {\em
Curves on $K3$ surfaces and modular forms}, J. of Topology {\bf 3}
(2010), 937--996.

\bibitem{Nak} H. Nakajima, {\em Lectures on Hilbert schemes of points on
surfaces}, AMS: Providence, RI, 1999.

\bibitem{GO} G. Oberdieck, {\em Gromov-Witten invariants
of the Hilbert scheme of points of a $K3$ surface}, arXiv:1406.1139.

\bibitem{OPvir} A. Okounkov and R. Pandharipande, {\em Virasoro
constraints for target curves}, Invent. Math. {\bf 163} (2006), 47--108. 

\bibitem{OPH} A. Okounkov and R. Pandharipande, {\em Quantum cohomology 
of the Hilbert scheme of points of the plane}, Invent. Math. {\bf 179} (2010),
523--557.

\bibitem{OPLC} A. Okounkov and R. Pandharipande, {\em The local Donaldson-Tho\-mas theory of curves}, Geom. Topol. {\bf 14} (2010), 1503--1567.

\bibitem{PaPix1} R. Pandharipande and A. Pixton, {\em Gromov-Witten/Pairs
descendent correspondence for toric 3-folds}, Geom. and Top. (to appear).

\bibitem{PaPix2} R. Pandharipande and A. Pixton, {\em Gromov-Witten/Pairs
correspondence for the quintic 3-fold}, arXiv:1206.5490.

\bibitem{PT1}
R.~Pandharipande and R.~P. Thomas,
\newblock {\em Curve counting via stable pairs in the derived category},
Invent. Math. \textbf{178} (2009), 407--447.
\bibitem{PT2}
R.~Pandharipande and R.~P. Thomas,
\newblock {\em The Katz-Klemm-Vafa conjecture for $K3$ surfaces}, arXiv:1404.6698.




\bibitem{CTC} C. T. C. Wall, {\em On the orthogonal groups of unimodular quadratic forms},
 Math. Ann. {\bf 147} (1962), 328--338.

\bibitem{YZ} S.-T. Yau and E. Zaslow, {\em BPS states, string duality, and
nodal curves on $K3$}, Nucl. Phys. {\bf B457} (1995), 484--512.




\end{thebibliography}
\end{document}